\documentclass[12pt]{article}
\usepackage[doc]{optional}

 \newenvironment{dedication}
 {\vspace{6ex}\begin{quotation}\begin{center}\begin{em}}
        {\par\end{em}\end{center}\end{quotation}}

\usepackage{tikz}
\usepackage{mathtools}
\DeclarePairedDelimiter{\ceil}{\lceil}{\rceil}
\usepackage{hyperref}
\usepackage{xfrac}
\usepackage{nameref}
\usepackage{empheq}
\usepackage[titletoc]{appendix}
\usepackage{enumerate}
\usepackage[shortlabels]{enumitem}
\setlist[enumerate]{nosep}
\usepackage[doc]{optional}
\usepackage{xcolor}
\usepackage{float}
\usepackage{soul}
\usepackage{graphicx}
\definecolor{labelkey}{rgb}{0,0.08,0.45}
\definecolor{refkey}{rgb}{0,0.6,0.0}
\definecolor{Brown}{rgb}{0.45,0.0,0.05}
\definecolor{lime}{rgb}{0.00,0.8,0.0}
\definecolor{lblue}{rgb}{0.5,0.5,0.99}
\usepackage{siunitx}
\usepackage{scalerel}

 \usepackage{mathpazo}
\usepackage{nameref}
\newcommand{\mychoose}[2]{\bigl({{#1}\atop#2}\bigr)}

\usepackage{amsthm}
\usepackage{amsmath}
\usepackage{scalerel}
\usepackage{stmaryrd}
\usepackage{amssymb}

\newcommand{\sepp}{\setlength{\itemsep}{-2pt}}

\newcommand{\aref}[1]{\hyperref[#1]{Appendix~\ref{#1}}}
\newcommand*{\tran}{^{\mkern-1.5mu\mathsf{T}}}

\newcommand{\nnn}{\ensuremath{{n\in{\mathbb N}}}}

\newcommand{\menge}[2]{\big\{{#1}~\big |~{#2}\big\}}

\newcommand{\Menge}[2]{\left\{{#1}~\Big|~{#2}\right\}}

\newcommand{\fenv}[1]%
{\ensuremath{\,\overrightarrow{\operatorname{env}}_{#1}}}
\newcommand{\benv}[1]%
{\ensuremath{\,\overleftarrow{\operatorname{env}}_{#1}}}

\newcommand{\infconv}{\ensuremath{\mbox{\small$\,\square \,$}}}

\newcommand{\exi}{\ensuremath{\exists\,}}

\newcommand{\RR}{\ensuremath{\mathbb R}}
\newcommand{\RP}{\ensuremath{\mathbb{R}_+}}

\newcommand{\NN}{\ensuremath{\mathbb N}}

\providecommand{\BB}[2]{\operatorname{ball}(#1;#2)}

\newcommand{\dom}{\ensuremath{\operatorname{dom}}}
\newcommand{\argmin}{\ensuremath{\operatorname{Argmin}}}

\newcommand{\sri}{\ensuremath{\operatorname{sri}}}

\newcommand{\ran}{\ensuremath{\operatorname{ran}}}
\newcommand{\zer}{\ensuremath{\operatorname{zer}}}

\newcommand{\Id}{\ensuremath{\operatorname{Id}}}

\newcommand{\Mbeta}[1]{M_{[#1]}}

\newcommand{\veet}{\ensuremath{{\scriptscriptstyle\vee}}} 

{\begin{list}{}{%
\settowidth{\labelwidth}{\textrm{#1~}}%
\setlength{\leftmargin}{\labelwidth+\labelsep}}}
{\end{list}}

\usepackage{amsthm}
\usepackage[capitalize]{cleveref}
\crefname{equation}{}{equations}
\crefname{chapter}{Appendix}{chapters}
\crefname{item}{}{items}

\newtheorem{theorem}{Theorem}[section]

\newtheorem{lem}[theorem]{Lemma}

\newtheorem{cor}[theorem]{Corollary}
\newtheorem{proposition}[theorem]{Proposition}
\newtheorem{prop}[theorem]{Proposition}

\newtheorem{defn}[theorem]{Definition}
\newtheorem{thm}[theorem]{Theorem}
\newtheorem{example}[theorem]{Example}
\newtheorem{ex}[theorem]{Example}
\newtheorem{fact}[theorem]{Fact}
\newtheorem{remark}[theorem]{Remark}
\newtheorem{rem}[theorem]{Remark}

\def\endproof{\ensuremath{\hfill \quad \blacksquare}}

\providecommand{\siff}{\Leftrightarrow}
\providecommand{\ds}{\displaystyle}

\providecommand{\abs}[1]{\lvert#1\rvert}

\providecommand{\norm}[1]{\lVert#1\rVert}
\providecommand{\Norm}[1]{\Big\lVert#1\Big\rVert}
\providecommand{\normsq}[1]{\lVert#1\rVert^2}
\providecommand{\bk}[1]{\left(#1\right)}
\providecommand{\stb}[1]{\left\{#1\right\}}
\providecommand{\innp}[1]{\langle#1\rangle}

\providecommand{\LA}{\Leftarrow}
\providecommand{\RA}{\Rightarrow}

\providecommand{\grad}{\nabla}

\providecommand{\lam}{\lambda}
\providecommand{\RR}{\mathbb{R}}

\providecommand{\ran}{\operatorname{ran}}

\providecommand{\dom}{\operatorname{dom}}

\newcommand{\fix}{\ensuremath{\operatorname{Fix}}}
\providecommand{\bdry}{\operatorname{bdry}}
\providecommand{\parl}{\operatorname{par}}

\providecommand{\gra}{\operatorname{gra}}
\providecommand{\Id}{\operatorname{{ Id}}}

\providecommand{\fady}{\varnothing}

\providecommand{\argmin}{\mathrm{arg}\!\min}

\providecommand{\rras}{\rightrightarrows}

\providecommand{\NN}{\mathbb{N}}
\providecommand{\BB}[2]{\operatorname{ball}(#1;#2)}

\providecommand{\fix}{\operatorname{Fix}}
\providecommand{\ran}{\operatorname{ran}}

\providecommand{\Id}{\operatorname{Id}}

\providecommand{\pt}{{\partial}}

\providecommand{\zer}{\operatorname{zer}}

\providecommand{\DR}{\operatorname{DR}}

\providecommand{\infconv}{\ensuremath{\mbox{\footnotesize$\,\infconv \,$}}}
\providecommand{\vee}{\ensuremath{{\scriptscriptstyle\vee}}} 

\providecommand{\TDR}[1][w,]{T_{{#1}{\DR}}}

\providecommand{\fady}{\varnothing}

\providecommand{\RR}{\mathbb{R}}
\providecommand{\NN}{\mathbb{N}}

\providecommand{\linop}{L}

\providecommand{\DR}{\operatorname{DR}}

\makeatletter
\def\namedlabel#1#2{\begingroup
   \def\@currentlabel{#2}%
   \label{#1}\endgroup
}
\makeatother

\definecolor{myblue}{rgb}{.8, .8, 1}
  \newcommand*\mybluebox[1]{%
    \colorbox{myblue}{\hspace{1em}#1\hspace{1em}}}

\allowdisplaybreaks 

\begin{document}

\title{\textsc
Affine nonexpansive operators,
Attouch-Th\'{e}ra duality
and the Douglas-Rachford algorithm}

\author{
Heinz H.\ Bauschke\thanks{
Mathematics, University
of British Columbia,
Kelowna, B.C.\ V1V~1V7, Canada. E-mail:
\texttt{heinz.bauschke@ubc.ca}.},
~~Brett Lukens\thanks{3990 Lansdowne Road, Armstrong, B.C.\ V0E~1B3, 
Canada. E-mail: 
\texttt{brjl94@gmail.com.}}
~~and Walaa M.\ Moursi\thanks{
Mathematics, University of British Columbia, Kelowna, B.C.\ V1V~1V7, Canada,
and 
Mansoura University, Faculty of Science, Mathematics Department, 
Mansoura 35516, Egypt. 
E-mail: \texttt{walaa.moursi@ubc.ca}.}}

\date{March 30, 2016}

\maketitle

\begin{dedication}
\vspace{-2.5 cm}
{In tribute to Michel Th\'era on his 70th birthday}
\end{dedication}

\begin{abstract}
\noindent
The Douglas-Rachford splitting algorithm was originally proposed
in 1956 to solve a system of linear equations arising from the
discretization of a partial differential equation.
In 1979, Lions and Mercier brought forward a very powerful
extension of this method suitable to solve optimization problems. 

In this paper, we revisit the original affine setting.
We provide a powerful convergence result for finding a zero of
the sum of two maximally monotone affine relations. As a by
product of our analysis, we obtain results concerning the
convergence of iterates of affine nonexpansive mappings as well
as Attouch-Th\'era duality. 
Numerous examples are presented. 

\end{abstract}
{\small

}

\section{Introduction}

Throughout this paper
\begin{empheq}[box=\mybluebox]{equation*}
\label{T:assmp}
X \text{~~is a real Hilbert space},
\end{empheq} 
with inner product $\innp{\cdot,\cdot}$ and
induced norm $\norm{\cdot}$. 
A central problem in optimization is to 
\begin{equation}
\label{e:sumprob}
\text{find $x\in X$ such that $0\in (A+B)x$,}
\end{equation}
where $A$ and $B$ are maximally monotone operators on $X$;
see, e.g., 
\cite{BC2011}, 
\cite{Borwein50}, 
\cite{Brezis}, 
\cite{BurIus},
\cite{Comb96},
\cite{Simons1},
\cite{Simons2},
\cite{Rock98},
\cite{Zeidler2a},
\cite{Zeidler2b}, and the references therein.
As Lions and Mercier observed in the their landmark paper
\cite{L-M79}, one may iteratively solve
the \emph{sum problem} \eqref{e:sumprob} by the
celebrated \emph{Douglas-Rachford splitting algorithm} 
(see also \cite{EckBer}).
This algorithm proceeds by iterating the operator
$T = \Id-J_A + J_BR_A$; the sequence $(J_AT^nx)_\nnn$ converges
to a solution of \eqref{e:sumprob} (see \cref{s:DRA} for details). 
The Douglas-Rachford algorithm was originally proposed in 1956 by
Douglas and Rachford \cite{DR56}. 
It can be viewed as a method for solving a system of linear
equations where the underlying coefficient matrix is positive
definite. 
The far-reaching extension to optimization 
provided by Lions and Mercier \cite{L-M79} is not at all obvious (for
the sake of completeness, we sketch this 
connection in the Appendix). 

In this paper, we concentrate on the affine setting.
In the original setting considered by Douglas and Rachford,
the operators $A$ and $B$ correspond to positive definite
matrices.
We extend this result in various directions. Indeed,
we obtain \emph{strong convergence} in possibly
\emph{infinite-dimensional} Hilbert space; the operators $A$ and
$B$ may be \emph{affine maximally monotone relations}, and we also
identify the \emph{limit}. 
The remainder of this paper is organized as follows.
In \cref{nexp:f:nexp}, we provide several results which will be
useful in the derivation of the main results. 
A new characterization of strongly convergent iterations of
affine nonexpansive operators 
(\cref{Thm:asym:Lin:Af}) is presented in \cref{itaff}. 
We also discuss when the convergence is linear.
In \cref{s:AT}, 
we obtain new results, which are formulated using the Douglas-Rachford
operator, on the relative geometry of 
the primal and dual (in the sense of Attouch-Th\'era duality) 
solutions to \eqref{e:sumprob}. 
The main algorithmic result (\cref{cor:DR}) is derived in 
\cref{s:DRA}. It provides precise
information on the behaviour of the Douglas-Rachford 
algorithm in the affine case. 
Numerous examples are presented in \cref{s:ex} where we also pay
attention to the
tridiagonal Toeplitz matrices and Kronecker products. 
In the Appendix, we sketch the connection between the historical
Douglas-Rachford algorithm and the powerful extension provided by
Lions and Mercier. 

Finally, the notation we employ is quite standard and follows
largely \cite{BC2011}. 
Let $C$ be a nonempty closed convex subset of 
$X$. We use $N_C$ and $P_C$ to denote 
the \emph{normal cone operator} and the \emph{projector}
associated with $C$, respectively.
Let $Y $ be a Banach space. We shall use 
$\mathcal{B}(Y)$ to denote the set of \emph{bounded
linear operators} on $Y$. Let $L\in \mathcal{B}(Y)$.
The \emph{operator norm} of $L$ is
 $\norm{L}=\sup_{\norm{y}\le1 }\norm{Ly}$. 
 Further notation is developed as necessary during the course of
 this paper. 

\section{Auxiliary results}
\label{nexp:f:nexp}

In this section, we collect various results that will be useful
in the sequel.

Suppose that $T:X\to X$. Then 
$T$ is \emph{nonexpansive} if 
\begin{equation}
(\forall x\in X)(\forall y\in X)\quad
\norm{Tx-Ty}\le \norm{x-y};
\end{equation}
$T$ is \emph{firmly nonexpansive} if
\begin{equation}
(\forall x\in X)(\forall y\in X)
\quad \normsq{Tx-Ty}+\normsq{(\Id-T)x-(\Id-T)y}
\le \normsq{x-y};
\end{equation}
$T$ is \emph{asymptotically regular} if 
\begin{equation}
(\forall x\in X)\quad T^n x -T^{n+1}x\to 0.
\end{equation}

\begin{fact}
\label{F:av:Asr}
Let $T:X\to X$. 
Then 
\begin{equation}
\left.\begin{array}{c}
T~\text{firmly nonexpansive}\\
\fix T\neq \fady
\end{array}
\right\}
\quad 
\RA \quad T~\text{asymptotically regular}.
\end{equation}
\end{fact}
\begin{proof}
See \cite[Corollary~1.1]{Br-Reich77} or \cite[Corollary~5.16(ii)]{BC2011}.
\end{proof}
\begin{fact}
\label{F:lin:ar}
Let $\linop\colon X\to X$ be linear and nonexpansive,
and let  $x\in X$. Then 
\begin{equation}
L^n x\to P_{\fix L} x \siff L^n x-L^{n+1} x\to 0.
\end{equation}
\end{fact}
\begin{proof}
See 
\cite[Proposition~4]{Baillon76},
 \cite[Theorem~1.1]{Ba-Br-Reich78}, 
 \cite[Theorem~2.2]{BDHP03}
or \cite[Proposition~5.27]{BC2011}.
(We mention in passing that in \cite[Proposition~4]{Baillon76} 
the author proved the result for general 
odd nonexpansive mappings in Hilbert spaces
and in
 \cite[Theorem~1.1]{Ba-Br-Reich78},
the authors generalize the result to Banach spaces.)
\end{proof}

 \begin{defn}
 Let $Y$ be a real Banach space,
let $(y_n)_\nnn$ be a sequence in $Y$ and let $y_\infty\in Y$.
Then $(y_n)_\nnn$ 
 \emph{converges} to $y_\infty$, denoted $y_n\to y_\infty$,
 if $\norm{y_n-y_\infty}\to 0$. $(y_n)_\nnn$;
 \emph{converges} $\mu$\emph{-linearly} to $y_\infty$
 if $\mu\in \left[0,1\right[$ and there exists
 $M\ge 0$ such that\footnote{
By \cite[Remark~3.7]{BLPW2013}, 
this is equivalent to $(\exists M>0)(\exists N\in \NN)(\forall
n\geq N)$
$ \norm{y_n-y_\infty}\le M\mu^n$.}
 \begin{equation}
 \label{eq:mu:lin}
 (\forall \nnn)\quad \norm{y_n- y_\infty}\le M\mu^n.
 \end{equation}
$(y_n)_\nnn$ 
 \emph{converges linearly} to $y_\infty$ if
  there exists $\mu\in \left[0,1\right[$
  and $M\ge 0$ such that \cref{eq:mu:lin} holds.
 \end{defn}
 
 \begin{example}[\bf{convergence vs. pointwise 
 convergence of bounded linear operators}]
 \label{ex:conv:pw}
Let $Y$ be a real Banach space,
let $(L_n)_\nnn$ be a sequence in $ \mathcal{B}(Y)$,
 and let ${L_{\infty}}\in \mathcal{B}(Y)$.
 Then one says:
 \begin{enumerate}
 \item  
  \label{ex:conv:conv}
 $(L_n)_\nnn$ \emph{converges
 or converges uniformly}
 to 
 ${L_{\infty}}$ in $\mathcal{B}(Y)$ if
  $L_n\to {L_{\infty}}$ $(\text{in}~ \mathcal{B}(Y))$.
 \item 
   \label{ex:pw:pw}
 $(L_n)_\nnn$ \emph{converges pointwise} 
 to  ${L_{\infty}}$
  if $(\forall y\in Y)$
   $L_ny\to{L_{\infty}} y$
$(\text{in}~ Y)$.
 \end{enumerate}
 \end{example}
 \begin{remark}
 It is easy to see that the convergence of a sequence
 of bounded linear operators implies pointwise convergence; 
however, the converse is not true 
(see, e.g., \cite[Example~4.9-2]{Krey89}). 
 \end{remark}
  
 \begin{lem}
 \label{lem:rate:UBP}
 Let $Y$ be a real Banach space,
let
 $(L_n)_\nnn$ be a sequence in $ \mathcal{B}(Y)$,
let ${L_{\infty}}\in \mathcal{B}(Y)$,
and let $\mu\in \left]0, 1\right[$. Then
\begin{equation}
(\forall y\in Y)~
\linop_n y\to {L_{\infty}} y ~~\mu\text{-linearly} ~~(\text{in}~Y)
\siff \linop_n \to{L_{\infty}}  ~~
\mu\text{-linearly}~~(\text{in}~\mathcal{B}(Y)).
\end{equation}
\end{lem}
\begin{proof}
Let $ y\in Y$. 
``$\RA$":
Because $\linop_n y\to {L_{\infty}} y ~~\mu\text{-linearly} $,
there exists $ M_y\ge 0$ 
such that $(\forall\nnn)$ 
$\norm{(\linop_n -{L_{\infty}})y}\le \mu^nM_y $;
equivalently,
\begin{equation}
\Norm{\bk{\frac{\linop_n -{L_{\infty}}}{\mu^n}}y}
=\frac{\norm{(\linop_n -{L_{\infty}})y}}{\mu^n}\le M_y. 
\end{equation}
It follows from the Uniform Boundedness Principle
(see, e.g., \cite[4.7-3]{Krey89}) 
applied to the sequence 
$((\linop_n -{L_{\infty}})/\mu^n)_\nnn$
that
$(\exists M\ge 0)(\forall n\in \NN)$ 
$\|(\linop_n -{L_{\infty}})/\mu^n\|\le M$;
equivalently,
$\norm{\linop_n -{L_{\infty}}}\le M\mu^n$, as required.
``$\LA$": 
Since $\linop_n\to L_\infty$ $\mu$-linearly, we have
$(\exi M\geq 0)(\forall\nnn)$ $\|\linop_n-L_\infty\| \leq M
\mu^n$. 
Therefore, $(\forall\nnn)$ 
$\norm{L_n y-{L_{\infty}}y}\le\norm{L_n -{L_{\infty}}}\norm{y}
\le M\norm{y} \mu^n$.
\end{proof}

\begin{lem}
\label{lem:fd:iff:c}
Suppose that $X$ is finite-dimensional, 
let $(\linop_n)_\nnn$ 
be a sequence of linear 
nonexpansive operators on $X$
and let $L_{\infty}:X\to X$.
Then the following are equivalent:
\begin{enumerate}
\item
\label{lem:fd:iff:c:i}
$(\forall x\in X)~L_n x\to L_{\infty}x$.
\item
\label{lem:fd:iff:c:ii}
$L_n \to L_{\infty}$ pointwise $~(\text{in}~X)$, and 
$L_{\infty}$ is linear and nonexpansive.
\item
\label{lem:fd:iff:c:iii}
$L_n\to{L_{\infty}}~(\text{in}~\mathcal{B}(X))$.
\end{enumerate}
\end{lem}
\begin{proof}
The implications 
``\ref{lem:fd:iff:c:i}$\RA$\ref{lem:fd:iff:c:ii}''
and 
``\ref{lem:fd:iff:c:iii}$\RA$\ref{lem:fd:iff:c:i}'' 
are easy to verify. 
``\ref{lem:fd:iff:c:ii}$\RA$\ref{lem:fd:iff:c:iii}":
Suppose that $(x_n)_\nnn$ is a sequence in $X$
such that $(\forall \nnn)~\norm{x_n}=1$ and 
\begin{equation}
\label{eq:approx:limTn}
\norm{L_n -L_\infty}
-\norm{L_n x_n-L_\infty x_n}\to 0.
\end{equation} 
We can and do assume that
$x_n\to x_\infty$. 
Since ${L_{\infty}}$
and  $(L_n)_\nnn$ are linear and nonexpansive,
we have
$\norm{{L_{\infty}}}\le 1$ and 
$(\forall n\in \NN)$ $\norm{L_n}\le 1$.
Using the triangle inequality, we have 
$ 
\norm{L_n x_n-{L_{\infty}} x_n}
=\norm{(L_n-{L_{\infty}})(x_n-x_\infty)+(L_n-{L_{\infty}})x_\infty}
\le \norm{L_n-{L_{\infty}}}\norm{x_n-x_\infty}
+\norm{(L_n-{L_{\infty}})x_\infty}
\le 2\norm{x_n-x_\infty}+\norm{(L_n-{L_{\infty}})x_\infty}
\to 0+0=0$. Now combine with \cref{eq:approx:limTn}.
\end{proof}

 \begin{cor}
 \label{F:lin:rate:Rn}

Suppose that $X$ is finite-dimensional, let 
$L:X\to X$ be linear, and let ${L_{\infty}}:X\to X$ be such that 
$L^n\to {L_{\infty}}$ pointwise.
Then $L^n\to {L_{\infty}}$ linearly.
\end{cor} 
\begin{proof}
Combine \cref{lem:fd:iff:c}
 and \cite[Theorem~2.12(i)]{BBNFW15}.
\end{proof}

\section{Iterating an affine nonexpansive operator}

\label{itaff}

We begin with a simple yet useful result.

\begin{thm}\label{P:aff:v0}
Let $\linop\colon X\to X$ be linear, 
let $b\in X$, set
$T\colon X\to X\colon x\mapsto \linop x + b$,
and suppose that $\fix T\neq\fady$. 
Let 
$x\in X$.
Then the following hold:
\begin{enumerate}
\item\label{P:aff:i:v0}
$ b\in \ran(\Id-\linop)$.

\item\label{P:aff:ii:v0}
$(\forall n\in \NN)$ 
$T^nx=\linop^nx+\sum_{k=0}^{n-1}\linop^kb$.

\end{enumerate}
\end{thm}

\begin{proof}
\ref{P:aff:i:v0}:
$\fix T\neq \fady\siff$
$(\exists y\in X)$ $y=Ly+b$
$\siff b\in \ran (\Id-L)$.
\ref{P:aff:ii:v0}:
We prove this by induction (see also 
\cite[Theorem~3.2(ii)]{Siopt2016}). 
When $n=0$ or $n=1$ the conclusion is obviously true.
Now suppose that, for some $n\in \NN$, 
\begin{equation}
T^nx=\linop^n x+\sum_{k=0}^{n-1}\linop^kb.
\end{equation}  
Then 
$T^{n+1}x=T(T^nx)
=T(\linop^n x+\sum_{k=0}^{n-1}\linop^kb)= 
\linop(\linop^n x+\sum_{k=0}^{n-1}\linop^kb)+b
=\linop^{n+1} x+\sum_{k=0}^{n}\linop^kb$.
\end{proof}

Let $S$ be a nonempty closed convex subset of
$X$ and let $w\in X$. We recall the following useful translation formula
(see, e.g., \cite[Proposition~3.17]{BC2011}):
\begin{equation}
\label{eq:proj:trans}
(\forall x\in X)\quad P_{w+S}x=w+P_S(x-w).
\end{equation}
\begin{lem}
\label{lem:shortcut}
Let $\linop\colon X\to X$ be linear and nonexpansive, 
let $b\in X$, set
$T\colon X\to X\colon x\mapsto \linop x + b$,
and suppose that $\fix T\neq\fady$. 
Then there exists a point $a\in X$
such that $b=a-La$  and
\begin{equation}
\label{e:0329a}
(\forall x\in X)\quad Tx=L(x-a)+a.
\end{equation} 
Moreover, the following hold:
\begin{enumerate}
\item
\label{lem:shortcut:i}
$\fix T=a+\fix L$.
\item
\label{lem:shortcut:ii}
$(\forall x\in X)\quad
P_{\fix T}x=a+P_{\fix L}(x-a)
=P_{(\fix L)^\perp}a+P_{\fix L}x$.
\item
\label{lem:shortcut:iii}
$(\forall \nnn)(\forall x\in X) \quad T^n x=a+L^n(x-a).$
\end{enumerate}
\end{lem}

\begin{proof}
The existence of $a$ and \eqref{e:0329a}
follows from \cref{P:aff:v0} and the linearity of $L$. 
\ref{lem:shortcut:i}:
Let $y\in X$. Then 
$y\in \fix T$ 
$\siff y-a \in \fix L\siff y\in a+\fix L$.
\ref{lem:shortcut:ii}:
The first identity follows from combining 
\ref{lem:shortcut:i} and \cref{eq:proj:trans}.
It follows from, e.g.,  \cite[Corollary~3.22(ii)]{BC2011}
that $a+P_{\fix L}(x-a)
=a+P_{\fix L}x-P_{\fix L}a=P_{(\fix L)^\perp}a+P_{\fix L}x$.
\ref{lem:shortcut:iii}:
By telescoping, we have 
\begin{equation}
\label{eq:telescop}
\sum_{k=0}^{n-1}\linop^kb=\sum_{k=0}^{n-1}\linop^k(a-La)=a-L^n a.
\end{equation}
Consequently, \cref{P:aff:v0}\ref{P:aff:ii:v0}
 and \cref{eq:telescop} yield 
$
 T^n x=L^n x+a-L^n a=a+L^n (x-a).
$ 
\end{proof}

The following result extends \cref{F:lin:ar}
from the linear to the affine case.

\begin{thm}
\label{Thm:asym:Lin:Af}
Let $\linop\colon X\to X$ be linear and nonexpansive, 
let $b\in X$, set
$T\colon X\to X\colon x\mapsto \linop x + b$, 
and suppose that $\fix T\neq\fady$. 
Then the following are equivalent:
\begin{enumerate}
\item
\label{Thm:asym:Lin:Af:i}
$L$ is asymptotically regular.
\item
\label{Thm:asym:Lin:Af:iii}
$L^n \to P_{\fix L} $ pointwise.
\item
\label{Thm:asym:Lin:Af:iv}
$T^n \to P_{\fix T} $ pointwise.
\item
\label{Thm:asym:Lin:Af:ii}
$T$ is asymptotically regular.
\end{enumerate}
\end{thm}
\begin{proof}
Let $x\in X$. 
``\ref{Thm:asym:Lin:Af:i}$\siff$\ref{Thm:asym:Lin:Af:iii}":
This is \cref{F:lin:ar}.
``\ref{Thm:asym:Lin:Af:iii}$\RA$\ref{Thm:asym:Lin:Af:iv}":
In view of \cref{lem:shortcut}\ref{lem:shortcut:iii}\&\ref{lem:shortcut:ii}
we have $T^n x=L^n (x-a)+a\to P_{\fix L}(x-a) +a=P_{\fix T}x$.
``\ref{Thm:asym:Lin:Af:iv}$\RA$\ref{Thm:asym:Lin:Af:ii}":
$T^n x-T^{n+1}x\to P_{\fix T} x-P_{\fix T} x=0$.
``\ref{Thm:asym:Lin:Af:ii}$\RA$\ref{Thm:asym:Lin:Af:i}":
Using \cref{lem:shortcut}\ref{lem:shortcut:iii}
 we have 
$ L^n x-L^{n+1} x=T^n (x+a)-T^{n+1} (x+a)\to 0$.
 \end{proof}

We now turn to linear convergence.

\begin{lem}
\label{lem:fd:iff}
Suppose that $X$ is finite-dimensional, 
and let $L:X\to X$ be linear and
nonexpansive. 
Then the following are equivalent:
\begin{enumerate}
\item
\label{lem:fd:iff:0}
$L$ is asymptotically regular.
\item
\label{lem:fd:iff:i}
$L^n \to P_{\fix L} $ pointwise $(\text{in}~ X)$.
\item
\label{lem:fd:iff:ii}
$L^n\to P_{\fix L}$ $(\text{in} ~\mathcal{B}(X))$.
\item
\label{lem:fd:iff:iii}
$L^n \to P_{\fix L} $ linearly pointwise $(\text{in}~ X)$.
\item
\label{lem:fd:iff:iv}
$L^n \to P_{\fix L} $ linearly $(\text{in}~ \mathcal{B}(X))$.
\end{enumerate}

\end{lem}
\begin{proof}
``\ref{lem:fd:iff:0}$\siff$\ref{lem:fd:iff:i}":
This follows from \cref{F:lin:ar}.
``\ref{lem:fd:iff:i}$\siff$\ref{lem:fd:iff:ii}":
Combine \cref{lem:fd:iff:c} and \cref{F:lin:ar}.
``\ref{lem:fd:iff:ii}$\RA$\ref{lem:fd:iff:iv}":
Apply \cref{F:lin:rate:Rn} with ${L_{\infty}}$
replaced by $ P_{\fix L}$.
``\ref{lem:fd:iff:iv}$\RA$\ref{lem:fd:iff:ii}":
This is obvious.
``\ref{lem:fd:iff:iii}$\siff$\ref{lem:fd:iff:iv}":
Apply \cref{lem:rate:UBP} to the sequence $(L^n)_\nnn$
and use \cref{F:lin:ar}.
\end{proof}
 \begin{thm}
\label{cor:rate}
Let $\linop\colon X\to X$ be linear and nonexpansive, 
let $b\in X$, set
$T\colon X\to X\colon x\mapsto \linop x + b$
 and let $\mu\in \left]0,1\right[$.
Then the following are equivalent:
\begin{enumerate}
\item
\label{cor:rate:i}
$T^n \to P_{\fix T}$  $\mu$-linearly
 pointwise $(\text{in}~ X)$.
\item
\label{cor:rate:ii}
$L^n  \to P_{\fix L}$
 $\mu$-linearly  pointwise $(\text{in}~ X)$.
 \item
 \label{cor:rate:iii}
$L^n  \to P_{\fix L}$  $\mu$-linearly $(\text{in} ~\mathcal{B}(X))$.
\end{enumerate}
\end{thm}
\begin{proof}
``\ref{cor:rate:i}$\siff$\ref{cor:rate:ii}":
It follows from  
\cref{lem:shortcut}\ref{lem:shortcut:iii}\&\ref{lem:shortcut:ii}
that $T^n x-P_{\fix T}x=a+L^n(x-a)-(a+P_{\fix L}(x-a))=
L^n(x-a)-P_{\fix L}(x-a)\to 0$, by \cref{F:lin:ar}.
``\ref{cor:rate:ii}$\siff$\ref{cor:rate:iii}":
Combine \cref{lem:rate:UBP} and \cref{F:lin:ar}.
\end{proof}
 
 \begin{cor}
 \label{cor:rate:Rd}
Suppose that $X $
 is finite-dimensional.
 Let $\linop\colon X\to X$ be linear, nonexpansive
 and asymptotically regular,
let $b\in X$, set
$T\colon X\to X\colon x\mapsto \linop x + b$
and suppose that $\fix T\neq\fady$.  
 Then $T^n \to P_{\fix T}$
 pointwise linearly. 
 \end{cor}
 
\begin{proof}
 It follows from \cref{F:lin:ar} that $L^n \to P_{\fix L}$ 
 pointwise. 
 Consequently, by \cref{F:lin:rate:Rn}, 
 $L^n \to P_{\fix L}$ linearly. 
Now apply \cref{cor:rate}. 
\end{proof}

\section{Attouch-Th\'{e}ra duality}

\label{s:AT}

Recall that a possibly set-valued operator $A:X\rras X$
is \emph{monotone} if for any two points 
$(x,u) $ and $(y,v)$ in the \emph{graph}
of $A$, denoted $\gra A$, we have 
$\innp{x-y,u-v}\ge 0$; $A$ is \emph{maximally monotone}
if there is no proper extension of $\gra A$ 
that preserves the monotonicity of $A$.
The \emph{resolvent}\footnote{It is well-known 
for
 a maximally monotone operator $A:X\rras X$ that 
 $J_A$ 
 is firmly nonexpansive and $R_A$ is nonexpansive 
 (see, e.g. \cite[Corollary~23.10(i)~and~(ii)]{BC2011}).}
of $A$, denoted by $J_A$, is defined by
$J_A=(\Id+A)^{-1}$ while the \emph{reflected resolvent}
of $A$ is $R_A=2J_A-\Id$.
In the following, we assume that
\begin{empheq}[box=\mybluebox]{equation*}
A:X\rras X  \text{~and ~} B:X\rras X 
\text{~are maximally monotone.}
\end{empheq} 
The \emph{Attouch-Th\'{e}ra} (see \cite{AT}) dual pair to the
primal pair $(A,B)$
is the pair\footnote{Let $A:X\rras X$. 
Then $A^\ovee=(-\Id)\circ A\circ (-\Id)$
and $A^{-\ovee}={(A^{-1}})^\ovee=(A^\ovee)^{-1}$.}
 $(A^{-1},B^{-\ovee})$.
 The \emph{primal} problem associated with 
$(A,B)$ is to 
 \begin{equation}
 \label{eq:def:prim:prom}
\text{find~} x\in X\text{~such that ~}0\in Ax+Bx, 
\end{equation}
and its
Attouch-Th\'{e}ra \emph{dual} problem is to
 \begin{equation}
 \label{eq:def:dual:prom}
\text{find~} x\in X\text{~such that ~}0\in A^{-1}x+B^{-\ovee}x. 
\end{equation}
We shall use $Z$ and $K$
to denote the sets of primal and dual
solutions of \cref{eq:def:prim:prom} and 
\cref{eq:def:dual:prom}
 respectively, i.e.,
\begin{equation}
\label{eq:def:ZK}
Z=Z_{(A,B)}=(A+B)^{-1}(0) \quad \text{and} 
\quad K=K_{(A,B)}=(A^{-1}+B^{-\ovee})^{-1}(0).
\end{equation}
The \emph{Douglas-Rachford} operator 
for the ordered pair $(A,B)$
(see \cite{L-M79})
is defined by
\begin{equation}
\label{e:defofDR}
T_{\DR}=T_{\DR}(A,B)=\Id-J_A+J_BR_A=\tfrac{1}{2}(\Id+R_BR_A).
\end{equation}
We recall that $C:X\rras X$ 
is \emph{paramonotone}\footnote{For a detailed discussion on
paramonotone operators we refer the reader to \cite{Iusem98}.}
 if 
it is monotone and $(\forall (x,u)\in \gra C)$
$(\forall (y,v)\in \gra C)$ we have
\begin{equation}
\left.
\begin{array}{c}
 (x,u)\in \gra C\\
(y,v)\in \gra C\\
\innp{x-y,u-v}=0
\end{array}
\right\}
\quad\RA\quad 
\big\{(x,v),(y,u)\big\}\subseteq \gra C.
\end{equation}

\begin{example}
\label{ex:para:goodsub}
Let $f:X\to\left]-\infty,+\infty\right]$ be 
proper, convex and lower semicontinuous.
Then $\pt f$ is paramonotone by 
\cite[Proposition~2.2]{Iusem98} 
(or by \cite[Example~22.3(i)]{BC2011}).
\end{example}

\begin{example}
\label{ex:not:para:badskew}
Suppose that $X=\RR^2$ and that
$A:\RR^2\to \RR^2:(x,y)\mapsto (y,-x)$.
Then one can easily verify that
 $A$ and $-A$ are maximally monotone but \emph{not} 
paramonotone by \cite[Section~3]{Iusem98}
(or \cite[Theorem~4.9]{BWY2014}).
\end{example}

\begin{fact}
\label{F:DR:FB}
The following hold: 
\begin{enumerate}
\item
\label{F:DR:FB:i}
$T_{\DR}$ is firmly nonexpansive.
\item
\label{F:DR:FB:ii}
$\zer(A+B)=J_A(\fix T_{\DR})$.
\end{enumerate}
If $A$ and $B$ are paramonotone, then we have additionally:
 \begin{enumerate}
 \setcounter{enumi}{2}
\item
\label{F:DR:FB:iii}
$\fix T_{\DR}=Z+K$.
\item
\label{F:DR:FB:iv}
$(K-K)\perp (Z-Z)$.
\end{enumerate}

\end{fact}
\begin{proof}
\ref{F:DR:FB:i}:
See \cite[Lemma~1]{L-M79}, \cite[Corollary~4.2.1]{EckThesis}, 
or \cite[Proposition~4.21(ii)]{BC2011}.
\ref{F:DR:FB:ii}:
See \cite[Lemma~2.6(iii)]{Comb04} or \cite[Proposition~25.1(ii)]{BC2011}.
\ref{F:DR:FB:iii}: See \cite[Corollary~5.5(iii)]{JAT2012}.
\ref{F:DR:FB:iv}: See \cite[Corollary~5.5(iv)]{JAT2012}.
\end{proof}

\begin{lem}
\label{lem:Brett:pw}
Suppose that $A$ and $B$ are paramonotone.
Let $k\in K$ be such that $(\forall z\in Z)$
$J_A(z+k)=P_Z(z+k)$. Then $k\in (Z-Z)^\perp$.
\end{lem}
\begin{proof}
By \cref{F:DR:FB}\ref{F:DR:FB:iii}, 
$\fix T_{\DR}=Z+K$. 
Let $z_1$ and $z_2$ be in $Z$. 
It follows from \cite[Theorem~4.5]{JAT2012}
that $(\forall z\in Z)~J_A(z+k)= z $.
Therefore, 
 \begin{equation}
 (\forall i\in\{1,2\})
\quad z_i+k\in \fix T_{\DR}
\;\text{and}\;
 z_i =J_A(z_i+k)
=P_Z(z_i+k). 
 \end{equation}
Furthermore, the Projection Theorem
 (see, e.g., \cite[Theorem~3.14]{BC2011}) yields
 \begin{equation}
 \innp{k,z_1-z_2}= \innp{z_1+k-z_1,z_1-z_2}
 =\innp{z_1+k-P_Z (z_1+k),P_Z(z_1+k)-z_2}\ge 0.
 \end{equation}
On the other hand, interchanging the roles
of $z_1$ and $z_2$ yields $ \innp{k,z_2-z_1}\ge 0$.
 Altogether, $ \innp{k,z_1-z_2}=0$.
\end{proof}

The next result relates the Douglas-Rachford operator to
orthogonal properties of primal and dual solutions. 

\begin{thm}
\label{prop:Brett:result}
Suppose that $A$ and $B$ are paramonotone.
Then the following are equivalent:
\begin{enumerate}
\item
\label{prop:Brett:result:iii}
$J_AP_{\fix T_{\DR}}=P_Z$.
\item
\label{prop:Brett:result:i}
${J_A}{ |}_{_{\fix T_{\DR}}}={P_Z}{|}_{_{\fix T_{\DR}}}$.
\item
\label{prop:Brett:result:ii}
$K\perp (Z-Z)$.
\end{enumerate}
\end{thm}
\begin{proof}
``\ref{prop:Brett:result:iii}$\RA$\ref{prop:Brett:result:i}": 
This is obvious.
``\ref{prop:Brett:result:i}$\RA$\ref{prop:Brett:result:ii}": 
Let $k\in K$ and let $z\in Z$. 
Then $\fix T_{\DR}=Z+K$ by 
\cref{F:DR:FB}\ref{F:DR:FB:iii};
hence, $z+k\in \fix T_{\DR}$.
Therefore $J_A(z+k)=P_Z(z+k)$.
Now apply \cref{lem:Brett:pw}.
``\ref{prop:Brett:result:ii}$\RA$\ref{prop:Brett:result:iii}":  This follows from
\cite[Theorem~6.7(ii)]{JAT2012}.
\end{proof}

\begin{cor}
\label{cor:aff:subsp:ZK}
Let \footnote{Let $C$ be nonempty closed 
convex subset of $X$.
Then $J_{N_C}=P_C$ by, e.g., \cite[Example~23.4]{BC2011}.}
 $U$ be a closed affine subspace of 
$X$, suppose that $A=N_U$ and that $B$ is paramonotone
such that $Z\neq \fady$.
Then the following hold\footnote{Suppose that
$U$ is a closed affine subspace of $X$.
We use $\parl U$ to denote the \emph{parallel space}
of $U$ defined by $\parl U=U-U$.}:
\begin{enumerate}
\item
\label{cor:aff:subsp:ZK:0}
$Z=U\cap (B^{-1}(\parl U)^\perp)\subseteq U$.
\item
\label{cor:aff:subsp:ZK:i}
$(\forall z\in Z)$ 
$K=(-Bz)\cap (\parl U)^\perp\subseteq (\parl U)^\perp$.
\item
\label{cor:aff:subsp:ZK:ii}
$K\perp (Z-Z)$.
\item
\label{cor:aff:subsp:ZK:iii}
$J_AP_{\fix T_{\DR}}=P_UP_{\fix T_{\DR}}=P_Z$.
\end{enumerate}
\end{cor}
\begin{proof}
Since $A=N_C=\partial \iota_C$, 
it is paramonotone by \cref{ex:para:goodsub}.
\ref{cor:aff:subsp:ZK:0}:
Let $x\in X$. Then $x\in Z \siff 0\in Ax+Bx=(\parl U)^\perp + Bx$
$\siff$ [$x\in U$ and there exists $y\in X$ such that
$y\in (\parl U)^\perp$ and $y\in Bx$] 
$\siff$ [$x\in U$ and there exists $y\in X$ such that
 $x\in B^{-1}y$ and $y\in (\parl U)^\perp$] 
 $\siff x\in U\cap B^{-1}((\parl U)^\perp)$.
 \ref{cor:aff:subsp:ZK:i}:
Let $z\in Z$. Applying \cite[Remark~5.4]{JAT2012}
to $(A^{-1}, B^{-\ovee})$ yields 
$K=(-Bz) \cap (Az)=(-Bz) \cap(\parl U)^\perp.$
\ref{cor:aff:subsp:ZK:ii}: 
By \ref{cor:aff:subsp:ZK:0} $Z-Z\subseteq U-U=\parl U$.
Now use \ref{cor:aff:subsp:ZK:i}. 
\ref{cor:aff:subsp:ZK:iii}:
Combine \ref{cor:aff:subsp:ZK:ii} and 
\cref{prop:Brett:result}.
\end{proof}

Using  \cite[Proposition~2.10]{JAT2012} we have 
\begin{equation}
\label{common:z}
\zer A\cap \zer B\neq \fady~ \siff~0\in K.
\end{equation}
\begin{thm}
\label{P:commz}
Suppose that $A$ and $B$ are 
paramonotone and that $\zer A\cap \zer B\neq\fady$.
Then the following hold:
\begin{enumerate}[]
\item 
\label{P:commz:i}
$Z=(\zer  A)\cap(\zer B)$ and $0\in K$.
\item
\label{P:commz:i:i}
$J_AP_{\fix T_{\DR}}=P_Z$.
\item
\label{P:commz:i:ii}
$K\perp(Z-Z)$.
\end{enumerate}
If, in addition, $A$ or
$B$ is single-valued, then we also have:
 \begin{enumerate}
 \setcounter{enumi}{3}
\item
\label{P:commz:ii}
$K=\stb{0}$.
\item
\label{P:commz:iii}
$\fix \TDR[]=(\zer  A)\cap(\zer B)
$.
\end{enumerate}
\end{thm}

\begin{proof}
\ref{P:commz:i}:
Since $\zer A\cap \zer B\neq\fady$,
it follows from \cref{common:z} 
that $0\in K$. Now apply 
\cite[Remark~5.4]{JAT2012} to get $Z=A^{-1}(0)\cap B^{-1}(0)=
(\zer  A)\cap(\zer B)$.  
\ref{P:commz:i:i}:
This is \cite[Corollary~6.8]{JAT2012}.
\ref{P:commz:i:ii}:
Combine \cref{common:z} 
and 
\cref{F:DR:FB}\ref{F:DR:FB:iv}.
\ref{P:commz:ii}: Let $C\in \stb{A,B}$ be single-valued.
Using \ref{P:commz:i} we have $Z\subseteq \zer C$.
Suppose that $C=A$ and let $z\in Z$. 
We use \cite[Remark~5.4]{JAT2012}
applied to $(A^{-1},B^{-\ovee})$
to
learn that $K=(Az)\cap(-Bz)$. Therefore
$\stb{0}\subseteq K\subseteq Az\subseteq A(\zer A)=\stb{0}$.
A similar argument applies if $C=B$.
\ref{P:commz:iii}:
Combine \cref{F:DR:FB}\ref{F:DR:FB:iii} with 
\ref{P:commz:i} \& \ref{P:commz:ii}.
\end{proof}

\begin{rem}
The conclusion of \cref{P:commz}\ref{P:commz:i}
generalizes the setting of convex feasibility problems.
Indeed, suppose that $A=N_U$ and $B=N_V$,
where $U $ and $V $ are nonempty closed convex subsets of
$X$ such that $U\cap V\neq\fady$.
Then $Z=U\cap V = \zer A\cap\zer B$. 
\end{rem}

The assumptions that $A$ and $B$ are paramonotone
are critical in the conclusion of \cref{P:commz}\ref{P:commz:i} 
as we illustrate now.
\begin{example} 
Suppose that 
$X=\RR^2$, that $U=\RR\times \stb{0}$,
that 
$A=N_U$ and that $B:\RR^2\to \RR^2:(x,y)\mapsto (-y,x)$,
is the counterclockwise rotator in the plane by $\pi/2$.
Then one verifies that
$\zer A=U$, $\zer B=\stb{(0,0)}$, $Z=\zer(A+B)=U$; however
$(\zer  A)\cap(\zer B)=\stb{(0,0)}\neq U=Z$.
Note that $A$ is paramonotone by
\cref{ex:para:goodsub} while $B$ is \emph{not} paramonotone 
by \cref{ex:not:para:badskew}.
\end{example}
In view of \cref{common:z} and \cref{P:commz}\ref{P:commz:i:i}, 
when $A$ and $B$ are paramonotone,
we have the implication 
$
0\in K\RA J_AP_{\fix T_{\DR}}=P_Z
$.
However the converse implication is
 not true, as we show in the next example.

\begin{example}
Suppose that $a\in X\smallsetminus\stb{0}$,
that $A=\Id-2a$ and that $B=\Id$. Then $Z=\stb{a}$,
$(A^{-1},B^{-\ovee})=(\Id+2a,\Id)$, hence $K=\stb{-a}$,
$Z-Z=\stb{0}$ and therefore $K\perp(Z-Z)$ which implies that
$J_AP_{\fix T_{\DR}}=P_Z$ by \cref{prop:Brett:result}, but
$0\not\in K$.
\end{example}

If neither $A$ nor $B$ is 
single-valued, then the conclusion of 
\cref{P:commz}\ref{P:commz:ii}\&\ref{P:commz:iii}
may fail as we now illustrate. 

\begin{example}
Suppose that $X=\RR^2$, 
that $U=\RR\times \stb{0}$,
that\footnote{Let $u\in X$ and let $r>0$.
We use
$\BB{u}{r}$ to denote the closed ball
in $X$
centred  at $u$ with radius $r$.
We also use $\RR_+$
to denote the set of nonnegative real numbers
$\left[0,+\infty\right[$.
} 
$V=\BB{(0,1)}{1}$, 
that
$A=N_U$ and that $B=N_V$.
By \cite[Example~2.7]{JAT2012} 
$Z=U\cap V=\stb{(0,0)}$ and 
$K=N_{\overline{U-V}}(0)=\RR_{+}\cdot(0,1)\neq \stb{(0,0)}$.
Therefore $\fix T_{\DR}=\RR_{+}\cdot(0,1)
\neq \stb{(0,0)}=U\cap V=\zer A
\cap \zer B$.
\end{example}

Recall that the Passty's \emph{parallel sum} (see e.g., \cite{Passty86} or \cite[Section~24.4]{BC2011})
is defined by
\begin{equation}
\label{eq:def:pars}
A\infconv B =(A^{-1}+B^{-1})^{-1}.
\end{equation}
In view of \cref{eq:def:ZK} and 
\cref{eq:def:pars}, one readily verifies that
\begin{equation}
\label{eq:desc:K}
K=(A\infconv  B^\ovee)(0). 
\end{equation}

\begin{lem}
\label{lem:B:lin}
Suppose that $B:X\rras X$ is linear\footnote{
$A\colon X\rras X$ is a \emph{linear relation} if 
$\gra A$ is a linear subspace of $X\times X$.}. 
Then the following hold:
\begin{enumerate}
\item
\label{lem:B:lin:i-}
$B^\ovee = B$ and $B^{-\ovee} = B^{-1}$. 
\item
\label{lem:B:lin:i}
$(A^{-1}, B^{-\ovee})=(A^{-1}, B^{-1})$.
\item
\label{lem:B:lin:ii}
$K=(A\infconv  B)(0)$.
\end{enumerate}
\end{lem}
\begin{proof}
This is straightforward from the definitions. 
\end{proof}

Let $f:X\to \left]-\infty,+\infty\right]$ be proper, 
convex and lower semicontinuous.
In the following we make use of the well-known identity\footnote{Let
$f:X\to \left]-\infty,+\infty\right]$ 
be proper, convex and lower semicontinuous.
We use $f^*$ to denote the \emph{convex conjugate}
(a.k.a. Fenchel conjugate) of $f$, defined by 
$f^*:X\to \left]-\infty,+\infty\right]:x\mapsto 
\sup_{u\in X}(\innp{x,u}-f(x))$.
} (see, e.g., 
\cite[Corollary~16.24]{BC2011}):
\begin{equation}
\label{eq:inv:sub:conj}
(\pt f)^{-1}=\pt f^*.
\end{equation}

\begin{cor}[{\bf subdifferential operators}]
\label{cor:subd}
Let $f:X\to \left]-\infty,+\infty\right]$ and 
$g:X\to \left]-\infty,+\infty\right]$
be proper, convex and lower semicontinuous.
Suppose that $A=\pt f $ and  that $B=\pt g$. 
 Then the following hold\footnote{Let $f:X\to \left]-\infty,+\infty\right]$.
 Then $f^\veet:X\to \left]-\infty,+\infty\right]:x\mapsto f(-x)$.
 }:
 \begin{enumerate}
 \item
 \label{cor:subd:i}
 $Z=(\pt f^*\infconv \pt g^*)(0)$.
 \item
  \label{cor:subd:ii}
 $K=(\pt f\infconv  \pt g^\veet)(0)
$.
\item
 \label{cor:subd:iii}
Suppose that\footnote{Let $f:X\to \left]-\infty,+\infty\right]$
be proper. The \emph{set of minimizers of} $f$,
$\stb{x\in X~|~f(x)=\inf f(X)}$, is 
denoted by $\argmin f$.}
$\argmin f\cap \argmin  g\neq \fady$.
Then $Z=\pt f^*(0)\cap \pt g^*(0)$.
 \item
 \label{cor:subd:iv}
 Suppose that\footnote{Let $S$ be nonempty subset 
 of $X$. The \emph{strong relative interior} of $S$, denoted 
 by $\sri S$, is the interior with respect to the closed affine hull of 
 $S$.}
  $0\in \sri(\dom f-\dom g)$.
 Then\footnote{Let 
 $f:X\to \left]-\infty,+\infty\right]$ 
 and $g:X\to \left]-\infty,+\infty\right]$ be 
proper, convex and lower semicontinuous.
The \emph{infimal convolution} of $f$ and $g$,
denoted by $f\infconv  g$, 
is the convex function 
$f\infconv  g:X\to \RR:x\mapsto \inf_{x\in X}\bk{f(y)+g(x-y)}$.} 
 $Z=\pt( f^*\infconv  g^*)(0)$.
  \item
 \label{cor:subd:v}
 Suppose that $0\in \sri(\dom f^*+\dom g^*)$.
 Then 
 $K=\pt( f\infconv  g^\veet)(0)$.

 \end{enumerate}

\end{cor}

\begin{proof}
Note that $A$ and $B$ are paramonotone 
by \cref{ex:para:goodsub}.
\ref{cor:subd:i}:
Using \cref{eq:inv:sub:conj} and \cref{eq:def:pars}
we have $Z=(A+B)^{-1}(0)=(((\pt f)^{-1})^{-1}+((\pt g)^{-1})^{-1})^{-1}(0)=
((\pt f^*)^{-1}+(\pt g^*)^{-1})^{-1}(0)=(\pt f^*\infconv \pt g^*)(0)$.
\ref{cor:subd:ii}:
Observe that $(\pt g)^{-\ovee}=((\pt g)^{\ovee})^{-1}=(\pt g^\veet)^{-1}$.
Therefore using \cref{eq:def:pars} we  have
$K=((\pt f)^{-1}+((\pt g)^{\ovee})^{-1})^{-1}(0)
=((\pt f)^{-1}+(\pt g^\veet)^{-1})^{-1}(0)
=(\pt f\infconv  \pt g^\veet)(0) $.
\ref{cor:subd:iii}:
Using \cref{P:commz}\ref{P:commz:i},
Fermat's rule (see, e.g., \cite[Theorem~16.2]{BC2011}) 
and \cref{eq:inv:sub:conj} 
we have
$Z=(\zer A)\cap(\zer B)=\argmin f\cap \argmin  g
 =(\pt f)^{-1}(0)\cap (\pt g)^{-1}(0)=
\pt f^*(0)\cap \pt g^*(0)$.
\ref{cor:subd:iv}: Combine \ref{cor:subd:i} 
and \cite[Proposition~24.27]{BC2011}
applied to the functions $f^*$ and $g^*$.
\ref{cor:subd:v}: Combine \ref{cor:subd:ii} 
and \cite[Proposition~24.27]{BC2011}
applied to the functions $f$ and $g^\veet
$.
\end{proof}

\section{The Douglas-Rachford algorithm in the affine case}

\label{s:DRA}
In this section we assume\footnote{
$A\colon X\rras X$ is an \emph{affine relation} if 
$\gra A$ is affine subspace of $X\times X$, i.e.,
a translation of a linear subspace of $X\times X$. 
For further information of affine relations we refer the reader to 
\cite{BWY2010}.} that
\begin{empheq}[box=\mybluebox]{equation*}
\label{T:assmp}
A:X\rras X  \text{~and ~} B:X\rras X 
\text{~are maximally monotone and affine,}
\end{empheq} 
and that 
\begin{empheq}[box=\mybluebox]{equation}
\label{eq:Z:fixT}
Z=\stb{x\in X~|~0\in Ax+Bx}\neq \fady.
\end{empheq} 
Since the resolvents $J_A$ and $J_B$ are affine
(see \cite[Theorem~2.1(xix)]{BMW2012}), so is $T_{\DR}$.
\begin{thm}
\label{cor:DR}
Let $x\in X$.
Then the following hold:
\begin{enumerate}
\item
\label{eq:DR:lim}
$T_{\DR}^n x\to P_{\fix T_{\DR}} x$.
\item
\label{eq:lim:JA}
Suppose that
$A$ and $B$ are 
paramonotone such that 
$K\perp (Z-Z)$ (as is the case when\footnote{See 
\cref{P:commz}\ref{P:commz:i:ii}.
}
$A$ and $B$ are paramonotone and 
$(\zer A)\cap(\zer B)\neq \fady$).
Then $J_AT_{\DR}^n x\to P_{Z} x $.
\item
\label{eq:DR:lim:fd}
Suppose that $X$ is finite-dimensional.
Then $T_{\DR}^n x\to P_{\fix T_{\DR}} x$ linearly
 and $J_AT_{\DR}^n x\to J_AP_{\fix T_{\DR}} x$ linearly.
\end{enumerate}
\end{thm}
\begin{proof}
\ref{eq:DR:lim}:
Note that in view of \cref{F:DR:FB}\ref{F:DR:FB:ii}
and \cref{eq:Z:fixT} we have $\fix T_{\DR}\neq \fady$.
Moreover \cref{F:DR:FB}\ref{F:DR:FB:i}
and \cref{F:av:Asr} imply that $T_{\DR}$ is 
asymptotically regular. 
It follows from \cref{Thm:asym:Lin:Af} that \ref{eq:DR:lim} holds.
\ref{eq:lim:JA}: Use \ref{eq:DR:lim} and 
\cref{prop:Brett:result}.
\ref{eq:DR:lim:fd}:
The linear convergence of $(T_{\DR}^n x)_\nnn$
 follows from \cref{cor:rate:Rd}.
 The linear convergence of $(J_AT_{\DR}^n x)_\nnn$
is a direct consequence of 
the  linear convergence of $(T_{\DR}^n x)_\nnn$
and the fact that $J_A$ is (firmly) nonexpansive. 
\end{proof}
\begin{rem}
\cref{cor:DR} generalizes the convergence results for
the original Douglas-Rachford
algorithm \cite{DR56} 
from particular symmetric matrices/affine operators
on a finite-dimensional space to 
general affine relations defined on possibly infinite
dimensional spaces, 
while keeping strong and linear convergence of 
the iterates of the \emph{governing sequence} 
$(T^n_{\DR}x)_\nnn$ 
and identifying the limit to be $P_{\fix T_{\DR}}x$.
Paramonotonicity coupled with common zeros yields convergence of
the shadow sequence $(J_AT^n_{\DR}x)_\nnn$ to $P_Zx$. 
\end{rem}

Suppose that $U$ and $V$ are 
nonempty closed convex subsets of 
$X$. Then 
\begin{equation}
T_{U,V}=T_{\DR}(N_U,N_V)=\Id-P_U+P_V(2P_U-\Id).
\end{equation}
\begin{proposition}
Suppose that $U$ and $V$ are closed linear subspaces of $X$.
Let $w\in X$.
Then $w+U$ and $w+V$ are closed affine subspaces of $X$,
$(w+U)\cap (w+V)\neq \fady$ and $(\forall n\in \NN)$
\begin{equation}
\label{ex:NC:trans}
T^n_{w+U,w+V}=T^n_{U,V}(\cdot-w)+w.
\end{equation}
\end{proposition}
\begin{proof}
Let $x\in X$.
We proceed by induction.
The case $n=0$ is clear.
We now prove the case
when $n=1$, i.e.,
\begin{equation}
\label{eq:n=1:case}
T_{w+U,w+V}=T_{U,V}(\cdot-w)+w. 
\end{equation}
Indeed,
$T_{w+U,w+V}x=(\Id-P_{w+U}+P_{w+V}(2P_{w+U}-\Id))x
=x-w-P_U(x-w)+w+P_V(2P_{w+U}x-x-w)=
x-w-P_U(x-w)+w+P_V(2w+2P_{U}(x-w)-x-w)
=(x-w)-P_U(x-w)+P_V(2P_U(x-w)-(x-w))+w=
(\Id-P_U+P_VR_U)(x-w)+w=T_{V,U}(x-w)+w $.
We now assume that
\cref{ex:NC:trans} holds for some $\nnn$.
Applying \cref{eq:n=1:case}
with $x$ replaced by $T^{n}_{w+V,w+U}x$ yields
\begin{align}
T^{n+1}_{w+U,w+V}x&=
T_{w+U,w+V}(T^{n}_{w+U,w+V}x)=T_{U,V}(T^{n}_{w+U,w+V}x-w)+w
\nonumber\\
&=T_{U,V}(T^{n}_{U,V}(x-w)+w-w)+w=T^{n+1}_{U,V}(x-w)+w;
\end{align}
hence \cref{ex:NC:trans} holds for all $\nnn$.
\end{proof}

\begin{example}[{\bf Douglas-Rachford in the
affine feasibility case }](see also \cite[Corollary~4.5]{JAT2014})
\label{ex:NC:app}
Suppose that $U$ and $V$ are closed linear subspaces of $X$.
Let $w\in X$ and let $x\in X$.
Suppose that $A=N_{w+U}$ and 
that $B=N_{w+V}$. Then 
$T_{w+U,w+V}x=Lx+b$, where $L=T_{U,V}$ and 
$b=w-T_{U,V}w$. Moreover,   
\begin{equation}
\label{eq:lim:NC}
T_{w+V,w+U}^n x\to P_{\fix T_{w+V,w+U}} x
\end{equation}
and 
\begin{equation}
\label{eq:lim:NC:JA}
J_AT_{w+V,w+U}^n x=P_{w+U}T_{w+V,w+U}^n x\to P_{Z} x=P_{(w+V)\cap(w+U)}x .
\end{equation}
Finally, if $U+V$ is closed (as is the 
case when $X$ is finite-dimensional)
then the 
convergence is linear with rate 
$c_F(U,V)<1$, where $c_F(U,V)$
is the cosine of the Friedrich's angle\footnote{
Suppose that $U$ and $V $ are closed linear subspaces of 
$X$.
The cosine of the Friedrichs angle is
$
c_F(\parl U,\parl V)=\sup_{\small\substack{u\in \parl U\cap W^\perp\cap{\BB{0}{1}} 
\\v\in \parl V\cap W^\perp\cap{\BB{0}{1}} }}
{\abs{\innp{u,v}}}<1,
$
where $W=\parl U\cap \parl V$.
}
between $U$ and $V$.
\end{example}
\begin{proof}
Using \cref{ex:NC:trans} with $n=1$ and the linearity of 
$T_{U,V}$ we have 
\begin{equation}
\label{eq:TUV:L:b}
T_{w+U,w+V}=T_{U,V}(\cdot-w)+w=T_{U,V}+w-T_{U,V}w.
\end{equation}
Hence $L=T_{U,V}$ and $b=w-T_{U,V}w$, as claimed.
To obtain  \cref{eq:lim:NC} and \eqref{eq:lim:NC:JA}, use 
\cref{cor:DR}\ref{eq:DR:lim} and 
\cref{cor:DR}\ref{eq:lim:JA}, respectively. 
The claim about the linear rate follows by combining
\cite[Corollary~4.4]{JAT2014} and 
\cref{cor:rate} with $T$ replaced by $T_{w+U,w+V}$,
$L$ replaced by $T_{U,V}$ and $b$ replaced by 
$w-T_{U,V}w$.
\end{proof}
\begin{rem}
When $X$ is infinite-dimensional, it is possible to
construct an example (see \cite[Section~6]{JAT2014}) 
of two linear subspaces
$U$ and $V$ where $c_F(U,V)=1$, and the rate
of convergence of $T_{\DR}$ is not linear. 
\end{rem}
The assumption that both operators be 
paramonotone is critical for the conclusion 
in \cref{cor:DR}\ref{eq:lim:JA}, as shown below.
\begin{example}
Suppose that $X=\RR^2$, that 
\begin{equation}
A=
\begin{pmatrix}
0&-1\\
1&0
\end{pmatrix}
\text{~and that~}B=N_{ \stb{0}\times\RR}.
\end{equation}
Then $T_{\DR}:\RR^2\to \RR^2:(x,y)\mapsto \tfrac{1}{2}(x-y)\cdot(1,-1)$,
 $\fix T_{\DR}=\RR\cdot(1,-1)$, $Z=\stb{0}\times\RR $,
 $K=\stb{0}$, hence $K\perp (Z-Z)$,
 and
$(\forall (x,y)\in \RR^2) (\forall n\ge 1)T_{\DR}^n(x,y)
=T_{\DR}(x,y)=\tfrac{1}{2}(x-y,y-x)
\in \fix T$,
however $(\forall (x,y)\in (\RR\smallsetminus \stb{0})\times \RR)$
$(\forall n\ge 1)$
\begin{equation}
(0,y-x)=J_A T_{\DR}^n(x,y)\neq P_Z(x,y)=(0,y).
\end{equation}
Note that $A$ is not paramonotone by 
\cref{ex:not:para:badskew}.
\end{example}
\begin{proof}
We have
\begin{equation}
J_A
=(\Id+A)^{-1}=
\begin{pmatrix}
1&-1\\
1&1
\end{pmatrix}^{-1}
=
\frac{1}{2}
\begin{pmatrix}
1&1\\
-1&1
\end{pmatrix},
\end{equation} 
 and 
\begin{equation}
R_A=2J_A-\Id=
\begin{pmatrix}
0&1\\
-1&0
\end{pmatrix}.
\end{equation}
Moreover, by \cite[Example~23.4]{BC2011}, 
 \begin{equation}
 J_B=P_{\RR\times\stb{0}}=\begin{pmatrix}
0&0\\
0&1
\end{pmatrix}.
 \end{equation}
 Consequently
 \begin{equation}
 T_{\DR}=\Id-J_A+J_BR_A
 =\begin{pmatrix}
1&0\\
0&1
\end{pmatrix}
-
\frac{1}{2}
\begin{pmatrix}
1&1\\
-1&1
\end{pmatrix}
+
\begin{pmatrix}
0&0\\
0&1
\end{pmatrix}
\begin{pmatrix}
0&1\\
-1&0
\end{pmatrix}
=\frac{1}{2}
\begin{pmatrix}
1&-1\\
-1&1
\end{pmatrix},
 \end{equation}
i.e.,
 \begin{equation}
 \label{eq:T:ex}
 T_{\DR}:\RR^2\to \RR^2:(x,y)\mapsto \tfrac{x-y}{2}(1,-1).
 \end{equation}
Now let $(x,y)\in \RR^2$. Then 
$(x,y)\in \fix T_{\DR}\siff (x,y)=(\tfrac{x-y}{2},-\tfrac{x-y}{2})$
$\siff x=\tfrac{x-y}{2}$ and $y=-\tfrac{x-y}{2}$
$\siff x+y=0$, hence $\fix T_{\DR}=\RR\cdot(1,-1)$ as claimed.
It follows from \cite[Lemma 2.6(iii)]{Comb04}
that $Z=J_A(\fix T_{\DR})=\RR\cdot J_A(1,-1)=\RR\cdot \tfrac{1}{2}(0,2)
=\stb{0} \times\RR $, as claimed.
Now let $(x,y)\in \RR^2$. By \cref{eq:T:ex} we have
$T_{\DR}(x,y)=\tfrac{x-y}{2}(1,-1)\in \fix T_{\DR}$,
hence $(\forall n\ge 1)$ $T_{\DR}^n (x,y)=T(x,y)=\tfrac{x-y}{2}(1,-1)$.
Therefore, 
$(\forall n\ge 1)$ $J_AT_{\DR}^n (x,y)=J_AT_{\DR}(x,y)=
J_A\bk{\tfrac{x-y}{2}(1,-1)}=(0,y-x)\neq (0,y)=
 P_Z(x,y)$ whenever $x \neq 0$.
\end{proof}

The next example illustrates that the assumption 
 $K\perp (Z-Z)$
is critical for the conclusion 
in \cref{cor:DR}\ref{eq:lim:JA}.
\begin{example}[{\bf when $K\not\perp(Z-Z)$}]
\label{ex:no:comm:zeros:Brett}
Let $u\in X\smallsetminus\stb{0}$.
Suppose that $A:X\to X:x\mapsto u$
and $B:X\to X:x\mapsto -u$. Then $A$ and 
$B$ are paramonotone, 
$A+B\equiv 0$ and therefore $Z=X$. 
Moreover, by \cite[Remark~5.4]{JAT2012} 
$(\forall z\in Z=X)$ $K=(Az)\cap(-Bz)
=\stb{u}\not \perp(Z-Z)=X$.
Note that $\fix T=Z+K=X+\stb{u}=X$ and 
$J_A:X\to X:x\mapsto x-u$. Consequently
\begin{equation}
(\forall x\in X)(\forall\nnn)\quad J_AT_{\DR}^n x=J_AP_{\fix T}x=J_Ax=x-u\neq x=P_Zx.
 \end{equation}
\end{example}

\begin{prop}[{\bf parallel splitting}]
Let $m\in \stb{2,3,\ldots}$,
and let $B_i:X\rras X$ be maximally monotone and affine,
$i\in\stb{1,2,\ldots,m}$,
such that $\zer(\sum_{i=1}^{m} B_i)\neq\fady$. 
Set ${\bf\Delta}=\stb{(x,\ldots,x)\in X^m~|~x\in X}$,
set ${\bf A}=N_{{\bf\Delta}}$,
 set ${\bf B}={\ds \times_{i=1}^m}B_i$,
set ${\bf T}=T_{\DR}({\bf A}, {\bf B})$,
 let $j:X\to X^m:x\mapsto(x,x,\ldots,x)$,
  and 
  let $e:X^m\to X:(x_1,x_2,\ldots,x_m)
  \mapsto\tfrac{1}{m}\bk{\sum_{i=1}^m x_i}$.  
 Let ${\bf x}\in X^m$.
 Then ${\bf \Delta}^\perp=\stb{(u_1,\ldots,u_m)
 \in X^m~|~\sum_{i=1}^{m}u_i=0}$,
 \begin{equation}
 \label{eq:paral:sp:ZK}
 {\bf Z}=Z_{({\bf A},{\bf B})}
 =j(\zer\bk{\sum\nolimits_{i=1}^{m} B_i})\subseteq {\bf \Delta}
 \quad
 \text{and}
 \quad
  {\bf K}=K_{({\bf A},{\bf B})}=(-{\bf B}( {\bf Z}) )
  \cap {\bf \Delta}^\perp \subseteq {\bf \Delta}^\perp.
  \end{equation}
Moreover, the following hold:
\begin{enumerate}
\item
\label{eq:DR:pars}
${\bf T}^n {\bf x}\to P_{\fix {\bf T}} {\bf x}$.
\item
\label{eq:DR:pars:fd}
Suppose that $X$ is finite-dimensional.
Then ${\bf T}^n {\bf x}\to P_{\fix {\bf T}} {\bf x}$ linearly and
$J_{\bf A}{\bf T}^n {\bf x}=P_{\bf \Delta}{\bf T}^n {\bf x}
\to P_{\bf \Delta}P_{\fix {\bf T}} {\bf x}$ linearly.
\item
\label{eq:lim:JA:pars}
Suppose that
$B_i:X\rras X$,
$i\in\stb{1,2,\ldots,m}$,  are
paramonotone. 
Then ${\bf B}$ is paramonotone
and $J_{\bf A}{\bf T}^n {\bf x}=
P_{\bf \Delta}{\bf T}^n {\bf x}\to P_{{\bf Z}} {\bf x} $.
Consequently, 
$e(J_{\bf A}{\bf T}^n {\bf x})=
e(P_{\bf \Delta}{\bf T}^n {\bf x})\to e(P_{{\bf Z}} {\bf x})\in Z $.
\end{enumerate}
\end{prop}
\begin{proof}
The claim about ${\bf \Delta}^\perp$
and
first identity in \cref{eq:paral:sp:ZK} 
follows from  
\cite[Proposition~25.5(i)\&(vi)]{BC2011},
whereas the second identity in \cref{eq:paral:sp:ZK} 
follows from \cref{cor:aff:subsp:ZK}\ref{cor:aff:subsp:ZK:ii}
applied
to $({\bf A},{\bf B})$.
\ref{eq:DR:pars}:
Apply \cref{cor:DR}\ref{eq:DR:lim} to $({\bf A},{\bf B})$.
\ref{eq:DR:pars:fd}:
Apply \cref{cor:DR}\ref{eq:DR:lim:fd} to $({\bf A},{\bf B})$.
\ref{eq:lim:JA:pars}: Let $({\bf x},{\bf u}), ({\bf y},{\bf v})$
be in $\gra {\bf B}$. On the one hand $\innp{{\bf x}-{\bf y}, {\bf u}-{\bf v}}=0$
$\siff \sum_{i=1}^{m}\innp{x_i-y_i,u_i-v_i}=0$, 
$(x_i,u_i), (y_i, v_i)$ are in $\gra B_i$, $i\in \stb{1,\ldots, m }$.
On the other hand, since $(\forall i\in \stb{1,\ldots, m }) ~B_i$ are monotone
we learn that $(\forall i\in \stb{1,\ldots, m }) ~\innp{x_i-y_i,u_i-v_i}\ge 0$.
Altogether, $(\forall i\in \stb{1,\ldots, m }) 
~\innp{x_i-y_i,u_i-v_i}= 0$.
Now use that paramonotonicity of $B_i$ to
deduce that $(x_i,v_i), (y_i, u_i)$ are 
in $\gra B_i$, $i\in \stb{1,\ldots, m }$; equivalently,
$({\bf x},{\bf v}), ({\bf y},{\bf u})$
in $\gra {\bf B}$. 
Finally, apply \cref{cor:aff:subsp:ZK}\ref{cor:aff:subsp:ZK:iii}.
\end{proof}

\section{Examples of linear monotone operators}

\label{s:ex}

In this section we present examples of monotone 
operators that are partly motivated by applications in partial differential
equations; see, e.g., \cite{Glow15} and \cite{Thomee90}. 
Let $M\in \RR^{n\times n}$.
Then we have the following equivalences:
\begin{subequations}
\label{e:0329b}
\begin{align}
M \text{~is monotone~} &\siff
\frac{M+M\tran}{2}~\text{is positive semidefinite}
\label{eq:ch:mono:symp}\\
&\siff
\text{the eigenvalues of\, $\displaystyle \frac{M+M\tran}{2}$\, lie in $\RP$.}
\label{eq:mono:ev}
\end{align}
\end{subequations}

\begin{lem}
\label{lem:mon:2by2:Brett}
Let \begin{equation}
M=
\begin{pmatrix}
\alpha &\beta\\
\gamma &\delta
\end{pmatrix}
\in \RR^{2\times 2}.
\end{equation}
Then
$M$ is monotone if and only if $\alpha\geq 0$, $\delta\geq 0$ 
and $4\alpha\delta\geq (\beta+\gamma)^2$.
\end{lem}
\begin{proof}
Indeed, the principal minors of 
$M+M\tran$ 
are 
$2\alpha$, $2\delta$ and $4\alpha\delta-(\beta+\gamma)^2$;
by, e.g., 
\cite[(7.6.12)~on~page~566]{Meyer2000}. 
\end{proof}

Note that if $M=M\tran$, then
$M$ is monotone if and only if the eigenvalues of $M$ lie in
$\RP$. If $M\neq M\tran$, then some information about the
location of the (possibly complex) eigenvalues of $M$ is
available:

\begin{lem}
\label{lem:re:part:+ve}
Let $M\in \RR^{n\times n}$ be monotone, 
and let $\stb{\lam_k}_{k=1}^n$ denote the set of 
eigenvalues of $M$. 
Then\footnote{Let $\mathbb{C}$ be
the set of complex numbers and let $z\in\mathbb{C}$.
We use $\operatorname{Re}(z)$ to refer to the real part
of the complex number $z$.} 
$\operatorname{Re}(\lam_k)\ge 0$ for
every $k\in \stb{1,\ldots, n}$.
\end{lem}
\begin{proof}
Write $\lam =\alpha+\mathrm{i}\beta$, where $\alpha$ and $\beta$
belong to $\RR$ and $\mathrm{i}=\sqrt{-1}$ and assume that 
$\lambda$ is an eigenvalue of
$M$ with (nonzero) eigenvector $w=u+\mathrm{i}v$, where
$u$ and $v$ are in $\RR^n$.
Then 
$(M-\lam \Id)w=0$ $\siff$
$((M-\alpha\Id)-\mathrm{i}\beta\Id)(u+\mathrm{i}v)=0$
$\siff$ $(M-\alpha\Id) u+\beta v=0$
and $(M-\alpha\Id)v-\beta u=0$.
Hence 
\begin{subequations}
\label{eq:re:part:+ve}
\begin{align}
\innp{u,(M-\alpha\Id) u}+\beta\innp{u,v}&=0,\\
\innp{v,(M-\alpha\Id) v}-\beta\innp{v,u}&=0.
\end{align}
\end{subequations}
Adding \cref{eq:re:part:+ve} yields
$\innp{u,(M-\alpha\Id) u}+\innp{v,(M-\alpha\Id) v}=0$;
equivalently, 
$\innp{u,Mu}+\innp{v,Mv}-\alpha\normsq{u}-\alpha\normsq{v}=0$.
Solving for $\alpha$ yields
\begin{equation}
\operatorname{Re}(\lam)=
\alpha=\frac{\innp{u,Mu}+\innp{v,Mv}}{\normsq{u}+\normsq{v}}\ge
0,
\end{equation}
as claimed. 
\end{proof}

  The converse of \cref{lem:re:part:+ve} is not true in
general, as we demonstrate in the following example.

\begin{example}
Let $\xi\in\RR\smallsetminus [-2,2]$, and set 
\begin{equation}
M=
\begin{pmatrix}
1&\xi\\
0&1
\end{pmatrix}.
\end{equation}
Then $M$ has $1$ as its only eigenvalue (with multiplicity $2$),
$M$ is not monotone
by \cref{lem:mon:2by2:Brett}, 
and $M$ is not symmetric. 
\end{example}

\begin{prop}
\label{prop:tri:mono:ch}

Consider the tridiagonal Toeplitz matrix 
\begin{equation} 
\label{eq:trid:def} 
M =
\begin{pmatrix}
\beta & \gamma & & 0\\
\alpha & \ddots & \ddots & \\
& \ddots & \ddots & \gamma \\
0 & & \alpha & \beta 
\end{pmatrix} \in \RR^{n\times n}.
\end{equation}
Then $M$ is monotone if and only if 
$\beta\ge \abs{\alpha+\gamma}\cos(\pi/(n+1))$. 
\end{prop}
\begin{proof}
Note that 
 \begin{equation}
\tfrac{1}{2}(M+M{\tran})=
\begin{pmatrix}
\beta & \tfrac{1}{2}(\alpha+\gamma) & & 0\\
\tfrac{1}{2}(\alpha+\gamma) & \ddots & \ddots & \\
& \ddots & \ddots &\tfrac{1}{2}(\alpha+\gamma)\\
0 & & \tfrac{1}{2}(\alpha+\gamma) & \beta 
\end{pmatrix}.
\end{equation}
By \cref{eq:ch:mono:symp},
$M$ is monotone 
$\siff $ $\tfrac{1}{2}(M+M{\tran})$ is positive semidefinite.
If $\alpha+\gamma= 0$ then $\tfrac{1}{2}(M+M{\tran})=\beta \Id$
 and therefore $\tfrac{1}{2}(M+M{\tran})$ 
 is positive semidefinite $\siff$
 $\beta\ge 0=\abs{\alpha+\gamma}$. 
 Now suppose that $\alpha+\gamma\neq 0$.
It follows from \cite[Example~7.2.5]{Meyer2000}
that the eigenvalues of $\tfrac{1}{2}(M+M{\tran})$ are 
\begin{equation}
\lambda_k=\beta+(\alpha+\gamma)
\cos\big(\tfrac{k\pi}{n+1}\big), 
\end{equation}
where $k\in \stb{1,\ldots,n}$.
Consequently, $\{\lam_k\}_{k=1}^{n}\subseteq \RP$ $\siff$ 
$\beta\ge \abs{(\alpha+\gamma)\cos(\pi/(n+1))}$.
Therefore, the characterization of monotonicity of $M$ follows
from \cref{eq:mono:ev}. 
\end{proof}

\begin{prop}
\label{prop:trid:inv}
Let 
\begin{equation} 
\label{eq:trid:def:rep} 
M =
\begin{pmatrix}
\beta & \gamma & & 0\\
\alpha & \ddots & \ddots & \\
& \ddots & \ddots & \gamma \\
0 & & \alpha & \beta 
\end{pmatrix}\in\RR^{n\times n}.
\end{equation}
Then exactly one of the following holds:
\begin{enumerate}
\item
\label{prop:trid:inv:i}
$\alpha\gamma=0 $ and $\operatorname{det} (M)=\beta^n$.
Consequently $M$ is invertible
$\siff \beta\neq 0$, in which case 
\begin{equation}
\label{eq:triang}
[M^{-1}]_{i,j}=(-\alpha)^{\max\stb{i-j,0}}
(-\gamma)^{\max\stb{j-i,0}} \beta^{\min\stb{j-i,i-j}-1}.
\end{equation}
\item
\label{prop:trid:inv:ii}
$\alpha\gamma\neq 0$.
Set $r=\tfrac{1}{2\alpha}(-\beta+\sqrt{\beta^2-4\alpha\gamma})$,
 $s=\tfrac{1}{2\alpha}(-\beta-\sqrt{\beta^2-4\alpha\gamma})$
  and $ {\Lambda}
  =\menge{\beta+2\gamma\sqrt{\alpha/\gamma} 
  \cos(k\pi/(n+1))}{k\in \stb{1,2,\ldots,n}}$.
  Then $rs\neq 0$. 
Moreover, 
$M$ is invertible\footnote{In the special case, when 
$\beta =0$, this is equivalent to saying that $M$
is invertible $\siff$ $n$ is even.}
$\siff 0\not\in   {\Lambda}$, in which case
\begin{subequations}
\label{eq:triang:rs}
\begin{align}
r\neq s&\RA
[M^{-1}]_{i,j}=-\frac{\gamma^{j-1}(r^{\min\stb{i,j}}-s^{\min\stb{i,j}})(r^{n+1}
s^{\max\stb{i,j}}-r^{\max\stb{i,j}}s^{n+1})}{\alpha^j(r-s)(r^{n+1}-s^{n+1})},
\\
r= s&\RA [M^{-1}]_{i,j}=
-
\frac{\gamma^{j-1}\min\stb{i,j}(n+1-\max\stb{i,j})r^{i+j-1}}{\alpha^j(n+1)}.
\end{align}
\end{subequations}
Alternatively, define the recurrence relations
\begin{subequations}
\label{subeq:recurr}
\begin{align}
u_0=0,&&u_1=1,&&u_k=-\tfrac{1}{\gamma}(\alpha u_{k-2}+\beta u_{k-1}),&& k\ge 2;\\
v_{n+1}=0,&&v_n=1,&&v_k=-\tfrac{1}{\alpha}(\beta v_{k+1}+\gamma v_{k+2}),&& k\le n-1.
\end{align}
\end{subequations}
Then 
\begin{equation}
\label{eq:triang:recur}
 [M^{-1}]_{i,j}=-\frac{u_{\min\stb{i,j}}v_{\max\stb{i,j}}}{v_0}\bk{\frac{\gamma}{\alpha}}^{j-1}. 
\end{equation}
\end{enumerate}
\end{prop}
\begin{proof}
\ref{prop:trid:inv:i}:
$\alpha\gamma=0\siff \alpha =0$ or $\gamma=0$,
in which case $M$ is a (lower or upper) triangular 
 matrix. Hence $\operatorname{det} (M)=\beta^n$, 
 and the characterization follows.
The formula in \cref{eq:triang}
is easily verified.
\ref{prop:trid:inv:ii}:
Note that $0\in \stb{r,s}\siff$
$\beta\in\{\pm\sqrt{\beta^2-4\alpha\gamma}\}$ 
$\siff \beta^2=\beta^2-4\alpha\gamma\siff \alpha\gamma=0$.
Hence $rs\neq 0$. 
Moreover, it follows 
from \cite[Example~7.2.5]{Meyer2000} that 
${\Lambda}$ is the set of eigenvalues of $M$; therefore, 
$M$ is invertible 
$\siff 0\not\in   {\Lambda}$.
The formulae \cref{eq:triang:rs}
follow from
\cite[Remark~2~on~page~110]{Yama97}.
The recurrence formulae defined in \cref{subeq:recurr}
and \cite[Theorem~2]{Yama97} yield
\cref{eq:triang:recur}.
\end{proof}

\begin{rem}
Concerning \cref{prop:trid:inv}, 
it follows from \cite[Section~2~on~page~44]{Torii66}
that we also have the alternative formulae
\begin{subequations}
\begin{align}
\label{eq:triang:rs:Torii}
r\neq s&\RA [M^{-1}]_{i,j}=\left\{
                \begin{array}{ll}
                  \displaystyle -\frac{1}{\gamma}\frac{s^{-i}-r^{-i}}{s^{-1}-r^{-1}}
                  \frac{s^{-n+j-1}-r^{-n+j-1}}{s^{-(n+1)}-r^{-(n+1)}},&j\ge i;\\
                  \\
             \displaystyle -\frac{1}{\alpha}\frac{s^{j}-r^{j}}{s-r} 
              \frac{s^{n-i+1}-r^{n-i+1}}{s^{n+1}-r^{n+1}},&j\le i,\\
                \end{array}
              \right . \\[+5mm]
\label{eq:triang:rr:Torri}
r= s&\RA [M^{-1}]_{i,j}=
\left\{
                \begin{array}{ll}
                  \displaystyle -\frac{i}{\gamma}\bk{1-\frac{j}{n+1}}r^{j-i+1}
                  ,&j\ge i;\\
                  \\
              \displaystyle -\frac{j}{\alpha} \bk{1-\frac{i}{n+1}}r^{j-i-1},&j\le i.\\
                \end{array}
              \right . 
              \end{align}
\end{subequations}
Using the binomial expansion, \cref{eq:triang:rs:Torii},
and a somewhat tedious calculation
which we omit here, one can show that $[M^{-1}]_{i,j}$ is equal
to 
\begin{equation}
-\tfrac{2\bk{\sum\limits_{m=0}^{\ceil{\sfrac{\min\stb{i,j}}{2}}-1}
{\mychoose{\min\stb{i,j}}{2m+1}} (-\beta)^{\min\stb{i,j}-(2m+1)} 
(\beta^2-4\alpha\gamma)^m}
      \bk{   \sum\limits_{m=0}^{\ceil{\sfrac{(n+1-\max\stb{i,j})}{2}}-1}
      {\mychoose{n-\max\stb{i,j}+1}{2m+1}}
         (-\beta)^{n-\max\stb{i,j}-2m} (\beta^2-4\alpha\gamma)^m}}
         {(2\alpha)^{\min\stb{0,j-i}} (2\gamma)^{\min\stb{0,i-j}}   
       \bk{ \sum\limits_{m=0}^{\ceil{\sfrac{(n+1)}{2}}-1}
       \mychoose{n+1}{2m+1}(-\beta)^{n-2m}(\beta^2-4\alpha\gamma)^m}}
         \end{equation}

provided that $r\neq s$. 
\end{rem}

\begin{example}
\label{ex:res:general}
Let $\beta\ge 2$,
set 
 \begin{equation} 
\label{eq:trid:lapl} 
M =
\begin{pmatrix}
\beta &-1 & & 0\\
-1 & \ddots & \ddots & \\
& \ddots & \ddots &-1\\
0 & &-1& \beta 
\end{pmatrix}\in\RR^{n\times n},
\end{equation}
set $r=\tfrac{1}{2}(\beta+\sqrt{\beta^2-4})$
 and set $s=\tfrac{1}{2}(\beta-\sqrt{\beta^2-4})$.
 Then $M$ is monotone and invertible.
Moreover, 
\begin{subequations}
\label{eq:triang:rs:lap:sp}
\begin{align}
r\neq s &\RA 
[M^{-1}]_{i,j}=\frac{(r^{\min\stb{i,j}}-s^{\min\stb{i,j}})(r^{n+1}
s^{\max\stb{i,j}}-r^{\max\stb{i,j}}s^{n+1})}{(r-s)(r^{n+1}-s^{n+1})},
\\
r= s &\RA  [M^{-1}]_{i,j}=\frac{
\min\stb{i,j}(n+1-\max\stb{i,j})}{n+1}.
\end{align}
\end{subequations}

Alternatively, define the recurrence relations
\begin{subequations}
\label{subeq:recurr:sym:sp}
\begin{align}
u_0=0,&&u_1=1,&&u_k=\beta u_{k-1}- u_{k-2},&& k\ge 2,\\
v_{n+1}=0,&&v_n=1,&&v_k=\beta v_{k+1}-v_{k+2},&& k\le n-1.
\end{align}
\end{subequations}
Then 
\begin{equation}
\label{eq:triang:recur:sym:sp}
 [M^{-1}]_{i,j}
 =-\frac{u_{\min\stb{i,j}}v_{\max\stb{i,j}}}{v_0}.
\end{equation}
\end{example}
\begin{proof}
The monotonicity of $M $ follows from
\cref{prop:tri:mono:ch}  by noting that 
$\beta\ge 2> 2\cos(\pi/(n+1))$. 
The same argument implies that
\begin{equation}
0\not\in{\Lambda}
  =\Menge{\beta-2
  \cos\big(\tfrac{k\pi}{n+1}\big)}{k\in \stb{1,2,\ldots,n}}.
  \end{equation}
Hence $M$ is invertible by \cref{prop:trid:inv}\ref{prop:trid:inv:ii}.
Note that $\beta=2\siff \beta^2-4=0\siff r=s=1$.
Now apply \cref{prop:trid:inv}\ref{prop:trid:inv:ii}.
\end{proof}

Let
$M_1=[\alpha_{i,j}]_{i,j=1}^{n}\in \RR^{n\times n}
$ and $M_2=[\beta_{i,j}]_{i,j=1}^{n}\in \RR^{n\times n}$.
Recall that the \emph{Kronecker product}
of $M_1$ and $M_2$ (see, e.g., \cite[page~407]{Lanc85}
or \cite[Exercise~7.6.10]{Meyer2000})
is defined by the block matrix
\begin{equation}
M_1\otimes M_2=[\alpha_{i,j}M_2]\in \RR^{n^2\times n^2}.
\end{equation}

\begin{lem}
\label{lem:kp:symm}
Let $M_1$ and $M_2$  be
symmetric matrices in $ \RR^{n\times n}$.
Then $M_1\otimes M_2\in  \RR^{n^2\times n^2}$
is symmetric.
\end{lem}
\begin{proof}
Using \cite[Exercise~7.8.11(a)]{Meyer2000} or 
\cite[Proposition~1(e)~on~page~408]{Lanc85}
we have $(M_1\otimes M_2){\tran}
=M_1{\tran}\otimes M_2{\tran}=M_1\otimes M_2$.
\end{proof}

The following fact is very useful in the conclusion of the upcoming
results.
\begin{fact}
\label{F:eigenv:prod}
Let $M_1$ and $M_2$ be in $\RR^{n\times n}$, 
with eigenvalues $\menge{\lam_k}{k\in\stb{1,\ldots, n}}$
 and $\menge{\mu_k}{k\in\stb{1,\ldots, n}}$. 
 Then the eigenvalues of $M_1\otimes M_2$
 are $\menge{\lam_j\mu_k}{j,k\in\stb{1,\ldots, n}}$.
 \end{fact}
\begin{proof}
See
\cite[Corollary~1~on~page~412]{Lanc85}
or \cite[Exercise~7.8.11(b)]{Meyer2000}.
\end{proof} 

\begin{cor}
\label{cor:sym:kp:mono}
Let $M_1$ and $M_2$ 
in $\RR^{n\times n}$ be monotone such that $M_1$ or $M_2$ 
is symmetric. 
Then $M_1\otimes M_2$ is monotone.
\end{cor}
\begin{proof}
According to \cref{e:0329b}, 
it is suffices to show that all the eigenvalues of  
$M_1\otimes M_2+(M_1\otimes M_2)\tran$
are nonnegative.
Suppose first that $M_1$ is symmetric.
Then using \cite[Proposition~1(e)\&(c)]{Lanc85} we
have $M_1\otimes M_2+(M_1\otimes M_2)\tran
=M_1\otimes M_2+M_1\otimes M_2\tran
=M_1\otimes(M_2+M_2\tran)$.
Since $M_2$ is monotone, 
it follows from \cref{e:0329b}
that all the eigenvalues of  $M_2+M_2\tran$
are nonnegative.
Now apply \cref{F:eigenv:prod} to $M_1$
 and $M_2+M_2\tran$  to learn that
  all the eigenvalues of 
 $M_1\otimes M_2+(M_1\otimes M_2)\tran$
are nonnegative, hence $M_1\otimes M_2$
 is monotone by \cref{eq:mono:ev}.
 A similar argument applies if $M_2$ is monotone.
\end{proof}

Note that the assumption that  at least one
matrix is 
symmetric is critical in 
\cref{cor:sym:kp:mono}, as we  show in the next example.
\begin{ex}
Suppose that 
\begin{equation}
M=
\begin{pmatrix}
0&-1\\
1
&0
\end{pmatrix}.
\end{equation}
Then $M$ is monotone,
with eigenvalues $\{\pm\mathrm{i}\}$, 
but \emph{not symmetric}. However, 
\begin{equation}
M\otimes M= 
\begin{pmatrix}
0&0&0&1\\
0&0&-1&0\\
0&-1&0&0\\
1&0&0&0
\end{pmatrix}.
\end{equation} 
is a symmetric matrix 
with eigenvalues $\stb{\pm 1}$
by \cref{F:eigenv:prod}.
Therefore 
$M\otimes M$ is 
 \emph{not monotone}
 by \cref{e:0329b}. 
\end{ex}

\begin{prop}
\label{prop:mono:M:kp:Id}
Let $M\in\RR^{n\times n}$ be symmetric.
Then  $\Id\otimes M$ is monotone
$\siff$ $M\otimes \Id$ is monotone
$\siff$ $M$ is monotone,
in which case we have
\begin{equation}
\label{eq:res:kp}
J_{\Id_n\otimes M}=\Id_{n}\otimes J_M\quad \text{and } 
\quad
J_{M\otimes \Id_n}=J_M\otimes\Id_{n}.
\end{equation}
\end{prop}
\begin{proof}
In view of \cref{F:eigenv:prod}
the sets of eigenvalues of $\Id\otimes M$,
$M\otimes \Id$, and 
$M$ coincide.
It follows from \cref{lem:kp:symm} that $\Id\otimes M$
 and $M\otimes \Id$ are symmetric.
Now apply
\cref{eq:mono:ev} and use the monotonicity of
$M$.
To prove \cref{eq:res:kp}, we use
\cite[Proposition~1(c)~on~page~408]{Lanc85}
to learn that $\Id_{n^2}+\Id_n\otimes M
=\Id_n\otimes \Id_n+\Id_n\otimes M=
\Id_n\otimes (\Id_n+M)$.
Therefore, by \cite[Corollary~1(b)~on~page~408]{Lanc85}
 we have 
 \begin{align}
 J_{\Id_n\otimes M}&=(\Id_{n^2}+\Id_n\otimes M)^{-1}
 =(\Id_n\otimes (\Id_n+M))^{-1}
 =\Id_n\otimes (\Id_n+M)^{-1}
 \nonumber\\
 &=\Id_{n}\otimes J_M.
 \end{align}
The other identity in \cref{eq:res:kp} is proved similarly.
\end{proof}

\begin{cor}
\label{cor:res:general}
Let $\beta\in \RR$. 
Set 
\begin{equation} 
\label{eq:trid:kp:ID}
\Mbeta{\beta} =
\begin{pmatrix}
\beta  & -1 & & 0\\
-1 & \ddots & \ddots & \\
& \ddots & \ddots & -1 \\
0 & & -1& \beta
\end{pmatrix}\in \RR^{n\times n},
\end{equation}
and let ${\bf M}_{\rightarrow}$ and 
 ${\bf M}_{\uparrow}$ be block matrices 
 in $\RR^{n^2\times n^2}$ defined by
\begin{equation}
 \label{eq:trid:def:rl:Kp}
 {\bf M}_{\rightarrow} =
\begin{pmatrix}
\Mbeta{\beta} & 0_{n} & & 0_{n}\\
0_{n} & \ddots & \ddots & \\
& \ddots & \ddots & 0_{n} \\
0_{n} & &0_{n} & \Mbeta{\beta}
\end{pmatrix}
 \quad\text{and}\quad
 {\bf M}_{\uparrow} =
\begin{pmatrix}
 \beta\Id_{n} & -\Id_{n} & & 0_{n}\\
-\Id_{n} & \ddots & \ddots & \\
& \ddots & \ddots & -\Id_{n} \\
0_{n} & & -\Id_{n} & \beta\Id_{n}
\end{pmatrix}
%
.
\end{equation}
Then ${\bf M}_{\rightarrow}=\Id_{n}\otimes \Mbeta{\beta}$ and
${\bf M}_{\uparrow}= \Mbeta{\beta}\otimes\Id_{n}$.
Moreover, 
${\bf M}_{\rightarrow}$ is monotone $\siff$ ${\bf M}_{\uparrow}$
is monotone $\siff$ $\Mbeta{\beta} $ is monotone
$\siff$
$\beta\ge 2\cos(\pi/(n+1))$, in which case
\begin{equation}
\label{eq:resol:beta}
J_{{\bf M}_{\rightarrow}}= \Id_{n}\otimes \Mbeta{\beta+1}^{-1}
\quad
\text{and}
\quad
J_{{\bf M}_{\uparrow}}=\Mbeta{\beta+1}^{-1} \otimes \Id_{n} .
\end{equation}
\end{cor}
\begin{proof}
It is straightforward to verify that
${\bf M}_{\rightarrow}=\Id_{n}\otimes \Mbeta{\beta}$
and
${\bf M}_{\uparrow}= \Mbeta{\beta}\otimes\Id_{n}$ .
It follows from \cref{prop:tri:mono:ch}
that $\Mbeta{\beta}$ is monotone 
$\siff $ 
$\beta\ge 2\cos\bk{\tfrac{\pi}{n+1}}$.
Now combine with \cref{prop:mono:M:kp:Id}.
To prove \cref{eq:resol:beta} note that 
$\Id_n+\Mbeta{\beta}=\Mbeta{\beta+1}$, and 
therefore $J_{\Mbeta{\beta}}=\Mbeta{\beta+1}^{-1}$.
The conclusion follows
by applying \cref{eq:res:kp}.
\end{proof}

The above matrices play a key role in the original design of the
Douglas-Rachford algorithm --- see the Appendix for details. 

\begin{prop}
\label{prop:mono:kp:M}
Let $n\in\{2,3,\ldots\}$, let 
$M\in \RR^{n\times n}$ and consider the block matrix
\begin{equation} 
\label{eq:trid:block}
{\bf M}=
\begin{pmatrix}
M & -\Id_{n} & & 0_{n}\\
-\Id_{n} & \ddots & \ddots & \\
& \ddots & \ddots & -\Id_{n} \\
0_{n} & & -\Id_{n} & M
\end{pmatrix}.
\end{equation}
Let $x=(x_1,x_2,\ldots,x_n)\in \RR^{n^2}$,
where $x_i\in \RR^n$, $i\in \stb{1,2,\ldots n}$.
Then
\begin{equation}
\label{eq:mono:kp:M}
\begin{split}
\innp{x,{\bf M} x}=
\innp{x_1,(M-\Id)x_1}+\sum_{k=2}^{n-1}\innp{x_k,(M-2\Id)x_k}\\
+\innp{x_n,(M-\Id)x_n}
+\sum_{i=1}^{n-1}\normsq{x_i-x_{i+1}}.
\end{split}
\end{equation}
Moreover the following hold:
\begin{enumerate}
\item 
\label{prop:mono:kp:M:0}
Suppose that $n=2$. Then
$M- \Id$ is monotone $ \siff {\bf M}$
is monotone.
\item 
\label{prop:mono:kp:M:i}
$M-2\Id$ is monotone $ \RA$ ${\bf M}$
is monotone. 
\item  
\label{prop:mono:kp:M:ii}
${\bf M}$
is monotone $\RA $ $M-2(1-\tfrac{1}{n})\Id$ is monotone $\RA$
$M$ is monotone.
\end{enumerate}
\end{prop}
\begin{proof}
We have
\begin{align}
\innp{x,{\bf M} x}
&= \innp{(x_1,x_2,\ldots,x_n),(Mx_1-x_2,-x_1+Mx_2-x_3, \ldots, -x_{n-1}+Mx_n )}
\label{eq:needed:seq}\nonumber\\
&= \innp{x_1, Mx_1}-\innp{x_1,x_2}-\innp{x_1,x_2}
+\innp{x_2, Mx_2}-\innp{x_2,x_3}\nonumber\\
&\qquad -\innp{x_2,x_3}+\innp{x_3, Mx_3}-\ldots
-\innp{x_2,x_3}+\innp{x_n, Mx_n}-\innp{x_{n-1},x_n}
\nonumber\\
&= \innp{x_1, Mx_1}-2\innp{x_1,x_2}
+\innp{x_2, Mx_2}-2\innp{x_2,x_3}\nonumber\\
&\qquad-\ldots-2\innp{x_2,x_3}+\innp{x_n, Mx_n}-2\innp{x_{n-1},x_n}\nonumber\\
&= \innp{x_1, Mx_1}+\normsq{x_1-x_2}-\normsq{x_1}-\normsq{x_2}+\innp{x_2, Mx_2}
+\normsq{x_2-x_3}
\nonumber\\
&\qquad-\normsq{x_2}-\normsq{x_3}+\ldots
+\normsq{x_{n-1}-x_n}-\normsq{x_{n-1}}-\normsq{x_n}+\innp{x_n, Mx_n}\nonumber\\
&= \innp{x_1, Mx_1}-\normsq{x_1}+\innp{x_2, Mx_2}-2\normsq{x_2}+\ldots
+\innp{x_n, Mx_n}-\normsq{x_n}+\ldots \nonumber\\
&\qquad +\normsq{x_1-x_2}+\normsq{x_2-x_3}+\ldots+\normsq{x_{n-1}-x_n}\nonumber\\
&=\innp{x_1,(M-\Id)x_1}+\bigg(\sum_{k=2}^{n-1}\innp{x_k,(M-2\Id)x_k}\bigg)
+\innp{x_n,(M-\Id)x_n}\nonumber\\
&\qquad+\sum_{i=1}^{n-1}\normsq{x_i-x_{i+1}}.\nonumber
\end{align}
\ref{prop:mono:kp:M:0}:
``$\RA$":
Apply \cref{eq:mono:kp:M} with $n=2$.
``$\LA$":
Let $y\in \RR^2$. Applying \cref{eq:mono:kp:M}
to the point $x=(y,y)\in \RR^4$,  we get 
$0\le \innp{x,{\bf M}x}=2\innp{y,(\Id-M)y}$.
\ref{prop:mono:kp:M:i}:
This is clear from \cref{eq:mono:kp:M}.
\ref{prop:mono:kp:M:ii}:
Let $y\in \RR^n$. Applying \cref{eq:mono:kp:M}
to the point $x=(y,y,\ldots, y)\in \RR^{n^2}$ yields
$0\le \innp{x,{\bf M}x}=2\innp{y,(M-\Id)y}+(n-2)\innp{y,(M-2\Id)y}=
\innp{y,(nM-2(n-1)\Id)y}$. 
Therefore, $M-2(1-\tfrac{1}{n})\Id$ is monotone. 
\end{proof}

The converse of \cref{prop:mono:kp:M}\ref{prop:mono:kp:M:i} 
is not true in general, as we illustrate now.
\begin{example}
Set 
\begin{equation}
M=
\begin{pmatrix}
1&-1\\
1
&1
\end{pmatrix},
\end{equation}
and let ${\bf M}$ be as defined in
\cref{prop:mono:kp:M}. 
Then one verifies easily that 
 ${\bf M}$ 
is monotone while 
$M-2\Id$ is not.
\end{example}

We now show that the converse of the implications in
 \cref{prop:mono:kp:M}\ref{prop:mono:kp:M:ii} 
are not true in general.

\begin{example}
Set $n=2$, set $M=\tfrac{1}{2}\Id\in\RR^{2\times 2}$, 
and let ${\bf M}$ be as defined in
\cref{prop:mono:kp:M}..
Then $M$ is monotone
but $M-2(1-\tfrac{1}{2})\Id=-\tfrac{1}{2}\Id$ is not monotone, 
and 
\begin{equation}
{\bf M}=
\frac{1}{2}
\begin{pmatrix}
1&0&-2&0\\
0&1&0&-2\\
-2&0&1&0\\
0&-2&0&1
\end{pmatrix}.
\end{equation}
Note the ${\bf M}$ is symmetric and has eigenvalues
$\stb{-\sfrac{1}{2},\sfrac{3}{2}}$, hence ${\bf M}$ is not monotone 
by \cref{e:0329b}.
\end{example}

\subsection*{Acknowledgments}
HHB was partially supported by the Natural Sciences 
and Engineering Research Council of Canada 
and by the Canada Research Chair Program.

\section*{Appendix}
In this section we briefly show the connection between the 
original Douglas-Rachford algorithm
introduced in \cite{DR56} (see also \cite{D55}, \cite{Milne53}
 and \cite{DR55} for variations of this method)
to solve certain 
types of heat equations
and the general algorithm 
introduced by Lions and Mercier
in \cite{L-M79} (see also \cite{Comb04}). 

Suppose that $\Omega$ is a bounded square region
in $\RR^2$.
Consider the Dirichlet
problem for the Poisson equation: 
Given $f$ and $g$, find $u:\Omega\to \RR$
such that  
\begin{empheq}[box=\mybluebox]{equation}
\label{eq:def:Dir:Lap}
 \Delta u=f \text{~on~}\Omega
 \quad\text{and}\quad u=g  \text{~on~}\bdry \Omega,
\end{empheq}
where  
$\Delta=\grad^2 =\frac{\pt^2 }{\pt x^2}+\frac{\pt^2 }{\pt y^2}$
is the Laplace operator and
 and $\bdry\Omega $ denotes 
 the boundary of $\Omega$.
Discretizing $u$ followed by converting it into a ``long vector"
$y$
(see \cite[Example~7.6.2~\&~Problem~7.6.9]{Meyer2000}) we obtain the system 
of linear equations
\begin{equation}
L_{\rightarrow}y+L_{\uparrow}y=-b.
\end{equation} 
Here $L_{\rightarrow}$ and $L_{\uparrow}$
denote the horizontal (respectively vertical) positive definite discretization
of the {negative} Laplacian over a square mesh with $n^2$
points at equally spaced intervals
(see, \cite[Problem~7.6.10]{Meyer2000}). We have 
\begin{equation}
L_{\rightarrow}=\Id\otimes M
\quad\text{and}\quad
L_{\uparrow}= M\otimes \Id, 
\end{equation}
where
\begin{equation} 
\label{eq:lap:build:block}
M =
\begin{pmatrix}
2  & -1 & & 0\\
-1 & \ddots & \ddots & \\
& \ddots & \ddots & -1 \\
0 & & -1&2
\end{pmatrix}\in \RR^{n\times n}.
\end{equation}
To see the connection to monotone operators,
set $A=L_{\rightarrow}$
and $B: L_{\uparrow}+b: y\mapsto L_{\uparrow}y+b$.
Then $A$ and $B$ are affine and 
strictly monotone.
The problem then reduces to 
\begin{equation}
\text{find $y\in \RR^{n^2}$ such
that  $Ay+By=0$},
\end{equation}
 and the algorithm proposed by
 Douglas and Rachford in \cite{DR56} 
becomes
\begin{subequations}
\begin{align}
y_{n+\sfrac{1}{2}}+Ay_n+By_{n+\sfrac{1}{2}}-y_n&=0,
\label{eq:sub:c:1}\\
y_{n+1}-y_{n+\sfrac{1}{2}}-Ay_{n}+Ay_{n+1}&= 0.
\label{eq:sub:c:2}
\end{align}
\end{subequations}
Consequently, 
\begin{subequations}
\begin{align}
\cref{eq:sub:c:1} 
&~\siff~ (\Id+B)(y_{n+\sfrac{1}{2}})
=(\Id-A)y_n
~\siff~
y_{n+\sfrac{1}{2}}=J_B(\Id-A)y_n,
\label{eq:dr:to:lm:1}\\
\cref{eq:sub:c:2} 
&~\siff~  (\Id+A) y_{n+1}=Ay_n+y_{n+\sfrac{1}{2}}
~\siff~ 
y_{n+1}=J_A(Ay_n+y_{n+\sfrac{1}{2}}).
\label{eq:dr:to:lm:2}
\end{align}
\end{subequations}
Substituting \cref{eq:dr:to:lm:1} into \cref{eq:dr:to:lm:2}
to eliminate $y_{n+\sfrac{1}{2}}$
yields
\begin{equation}
\label{eq:pre:ch:var}
y_{n+1}=J_A\big(Ay_n+J_B(\Id-A)y_n\big).
\end{equation}
To proceed further, we must show that
\begin{subequations}
\label{e:0329c}
\begin{align}
(\Id-A)J_A&=R_A\label{eq:sub:d:1}\\
AJ_A&=\Id-J_A.\label{eq:sub:d:2}
\end{align}
\end{subequations}
Indeed, note that
$\Id -A=2\Id-(\Id+A)$, therefore multiplying by $J_A=(\Id+A)^{-1}$
from the right yields
$
(\Id-A)J_A=(2\Id-(\Id+A))J_A=2J_A-\Id =R_A$.
Hence $J_A-AJ_A=J_A-(\Id-J_A)$;
equivalently, $AJ_A=\Id-J_A$.
Now consider the change of variable
\begin{equation}
\label{eq:change:var}
(\forall \nnn)\quad x_n=(\Id+A)y_n,
\end{equation}
which is equivalent to
$y_n=J_A x_n$.
Substituting \cref{eq:pre:ch:var} into \cref{eq:change:var},
 and using \cref{e:0329c}, 
yield
\begin{align}
x_{n+1}&=(\Id+A)y_{n+1}=(\Id+A)J_A(Ay_n+J_B(\Id-A)y_n)
=Ay_n+J_B(\Id-A)y_n\nonumber\\
&=AJ_A x_n+J_B(\Id-A)J_A x_n=x_n-J_Ax_n+J_BR_Ax_n=(\Id-J_A+J_BR_A)x_n,
\end{align}
which is the Douglas-Rachford update formula \eqref{e:defofDR}.

We point out that $J_A=J_{L_{\rightarrow}}$,
 and using \cite[Proposition~23.15(ii)]{BC2011} 
 we have $J_B=J_{L_{\uparrow}+b}=J_{L_{\uparrow}}-J_{L_{\uparrow}}b$.
To calculate
$J_A$ and $J_B$ 
apply \cref{cor:res:general} 
to
 get
 \begin{equation}
J_A=\Id_n\otimes J_M
 \quad\text{and}\quad
 J_B
 =J_M\otimes\Id_n-(J_M\otimes\Id_n)(b). 
\end{equation}
For instance, when $n=3$, the above calculations yield
\begin{equation}
J_M=
\begin{pmatrix}
 \frac{8}{21} & \frac{1}{7} & \frac{1}{21} \\[+1mm]
 \frac{1}{7} & \frac{3}{7} & \frac{1}{7} \\[+1mm]
 \frac{1}{21} & \frac{1}{7} & \frac{8}{21} \\[+1mm]
\end{pmatrix}, 
\end{equation}
\begin{equation}
\Id_3\otimes J_M = \begin{pmatrix}
 \frac{8}{21} & \frac{1}{7} & \frac{1}{21} & 0 & 0 & 0 & 0 & 0 &
 0 \\[+1mm]
 \frac{1}{7} & \frac{3}{7} & \frac{1}{7} & 0 & 0 & 0 & 0 & 0 & 0
 \\[+1mm]
 \frac{1}{21} & \frac{1}{7} & \frac{8}{21} & 0 & 0 & 0 & 0 & 0 &
 0 \\[+1mm]
 0 & 0 & 0 & \frac{8}{21} & \frac{1}{7} & \frac{1}{21} & 0 & 0 &
 0 \\[+1mm]
 0 & 0 & 0 & \frac{1}{7} & \frac{3}{7} & \frac{1}{7} & 0 & 0 & 0
 \\[+1mm]
 0 & 0 & 0 & \frac{1}{21} & \frac{1}{7} & \frac{8}{21} & 0 & 0 &
 0 \\[+1mm]
 0 & 0 & 0 & 0 & 0 & 0 & \frac{8}{21} & \frac{1}{7} &
 \frac{1}{21} \\[+1mm]
 0 & 0 & 0 & 0 & 0 & 0 & \frac{1}{7} & \frac{3}{7} & \frac{1}{7}
 \\[+1mm]
 0 & 0 & 0 & 0 & 0 & 0 & \frac{1}{21} & \frac{1}{7} &
 \frac{8}{21} \\[+1mm]
\end{pmatrix},
\end{equation}
and 
\begin{equation}
J_M\otimes \Id_3 = 
\begin{pmatrix}
 \frac{8}{21} & 0 & 0 & \frac{1}{7} & 0 & 0 & \frac{1}{21} & 0 &
 0 \\[+1mm]
 0 & \frac{8}{21} & 0 & 0 & \frac{1}{7} & 0 & 0 & \frac{1}{21} &
 0 \\[+1mm]
 0 & 0 & \frac{8}{21} & 0 & 0 & \frac{1}{7} & 0 & 0 &
 \frac{1}{21} \\[+1mm]
 \frac{1}{7} & 0 & 0 & \frac{3}{7} & 0 & 0 & \frac{1}{7} & 0 & 0
 \\[+1mm]
 0 & \frac{1}{7} & 0 & 0 & \frac{3}{7} & 0 & 0 & \frac{1}{7} & 0
 \\[+1mm]
 0 & 0 & \frac{1}{7} & 0 & 0 & \frac{3}{7} & 0 & 0 & \frac{1}{7}
 \\[+1mm]
 \frac{1}{21} & 0 & 0 & \frac{1}{7} & 0 & 0 & \frac{8}{21} & 0 &
 0 \\[+1mm]
 0 & \frac{1}{21} & 0 & 0 & \frac{1}{7} & 0 & 0 & \frac{8}{21} &
 0 \\[+1mm]
 0 & 0 & \frac{1}{21} & 0 & 0 & \frac{1}{7} & 0 & 0 &
 \frac{8}{21} \\[+1mm]
\end{pmatrix}. 
\end{equation}


{\small 
\begin{thebibliography}{999}
\sepp
\bibitem{AT}
H.\ Attouch and M.\ Th\'era,
A general duality principle for the sum of two operators,
\emph{Journal of Convex Analysis} 3 (1996), 1--24.

\bibitem{Baillon76}
J.B.\ Baillon,
Quelques propri\'et\'es de convergence asymptotique
pour les contractions impaires,
\emph{Comptes rendus de l'Acad\'emie des Sciences}
238(1976), Aii, A587-A590.

\bibitem{Ba-Br-Reich78}
J.B.\ Baillon, R.E.\ Bruck and S.\ Reich, 
On the asymptotic behavior of nonexpansive 
mappings and semigroups in Banach spaces, 
\emph{Houston Journal of Mathematics} 4 (1978), 1--9.

\bibitem{JAT2014}
H.H.\ Bauschke, J.Y.\ Bello Cruz, T.T.A.\ Nghia, H.M.\ Phan and X.\ Wang,
The rate of linear convergence of the Douglas-Rachford algorithm 
for subspaces is the cosine of the Friedrichs angle, 
\emph{Journal of Approximation Theory} 185 (2014), 63--79.

\bibitem{BBNFW15}
H.H.\ Bauschke, J.Y.\ Bello Cruz, T.T.A.\ Nghia, H.M.\ Phan and X.\ Wang,
Optimal rates of convergence of matrices with applications,
to appear in \emph{Numerical Algorithms}.
\texttt{DOI 10.1007/s11075-015-0085-4}


\bibitem{JAT2012} H.H.\ Bauschke, R.I.\ Bo\c{t}, 
W.L.\ Hare and W.M.\ Moursi,
Attouch--Th\'era 
duality revisited: paramonotonicity and operator splitting, 
\emph{Journal of Approximation Theory} 164 (2012), 1065--1084. 

\bibitem{BC2011}
H.H.\ Bauschke and P.L.\ Combettes,
\emph{Convex Analysis and Monotone 
Operator Theory in Hilbert Spaces},
Springer, 2011.

\bibitem{BDHP03}
H.H.\ Bauschke, F.\ Deutsch, H.\ Hundal 
and S.-H.\ Park, Accelerating the convergence of 
the method of alternating projections, 
\emph{Transactions of the AMS} 
355 (2003), 3433--3461.  

\bibitem{BMW2012}
H.H.\ Bauschke, S.M.\ Moffat, and X.\ Wang,
 Firmly nonexpansive mappings 
 and maximally monotone operators: correspondence and duality, 
 \emph{Set-Valued and Variational Analysis} 20 (2012), 131--153.

\bibitem{BLPW2013}
H.H.\ Bauschke, 
D.R.\ Luke,
 H.M.\ Phan and X.\ Wang, 
 Restricted normal cones and the 
 method of alternating projections: applications, 
 \emph{Set-Valued and Variational Analysis} 21 (2013), 475--501.

\bibitem{Siopt2016}
H.H.\ Bauschke and W.M.\ Moursi, On the Douglas-Rachford 
algorithm for two (not necessarily intersecting) affine subspaces, 
to appear in \emph{SIAM Journal on Optimization}.

\bibitem{BWY2010}
H.H.\ Bauschke, X.\ Wang and L.\ Yao,
 On Borwein-Wiersma decompositions of monotone linear relations, 
\emph{SIAM Journal on Optimization}~20 (2010), 2636--2652.

\bibitem{BWY2014}
H.H.\ Bauschke, X.\ Wang and L.\ Yao,
Rectangularity and paramonotonicity of maximally monotone operators, 
\emph{Optimization}~63 (2014), 487--504.



\bibitem{Borwein50}
J.M.\ Borwein,
Fifty years of maximal monotonicity,
\emph{Optimization Letters}~4 (2010), 473--490.

\bibitem{Brezis}
H.\ Brezis,
\emph{Operateurs Maximaux Monotones et
Semi-Groupes de Contractions dans les Espaces de Hilbert},
North-Holland/Elsevier, 1973. 

\bibitem{Br-Reich77}
R.E.\ Bruck and S.\ Reich, 
Nonexpansive projections and resolvents of accretive operators in
Banach spaces, \emph{Houston Journal of Mathematics} 
3 (1977), 459--470.


\bibitem{BurIus}
R.S.\ Burachik and A.N.\ Iusem,
\emph{Set-Valued Mappings and Enlargements
of Monotone Operators}, Springer-Verlag, 2008.

\bibitem{Comb96}
P.L.\ Combettes,
The convex feasibility problem in image recovery,
\emph{Advances in Imaging and Electron Physics}~25 (1995), 155--270.

\bibitem{Comb04}
P.L.\ Combettes,
Solving monotone inclusions via compositions of nonexpansive averaged
operators,
\emph{Optimization}~53 (2004), 475--504.


\bibitem{DR55}
J.\ Douglas and H.H.\ Rachford, 
The numerical solution of parabolic and elliptic
differential equations,
\emph{Journal of SIAM}~3 (1955), 28--41.

\bibitem{DR56}
J.\ Douglas and H.H.\ Rachford, 
On the numerical solution of the heat conduction problem
in 2 and 3 space variables,
\emph{Transactions of the AMS}~82 (1956), 421--439.


\bibitem{D55}
J.\ Douglas, 
On the numerical integration of 
$\frac{\pt^2 u}{\pt x^2}+\frac{\pt^2 u}{\pt y^2}=\frac{\pt u}{\pt t}  $
by implicit methods,
\emph{Journal of SIAM}~3 (1955), 42--65.
\bibitem{EckThesis}
J.\ Eckstein,
\emph{Splitting Methods for Monotone Operators with
Applications to Parallel Optimization},
Ph.D.~thesis, MIT, 1989.

\bibitem{EckBer}
J.\ Eckstein and D.P.\ Bertsekas,
On the Douglas--Rachford splitting method
and the proximal point algorithm for maximal monotone
operators,
\emph{Mathematical Programming}~55 (1992), 293--318.

\bibitem{Glow15}
R.\ Glowinski,
\emph{Variational Methods for the Numerical Solution of Nonlinear Elliptic Problems},
SIAM,
2015.


\bibitem{Iusem98}
A.N.\ Iusem,
On some properties of paramonotone operators,
\emph{Journal of Convex Analysis}~5 (1998), 269--278.


\bibitem{Krey89}
E.\ Kreyszig,
\emph{Introductory Functional Analysis with Applications},
Wiley, 1989.

\bibitem{Lanc85}
P.\ Lancaster and M.\ Tismenetsky,
\emph{The Theory of Matrices with Applications}, Academic Press,
1985.

\bibitem{L-M79}
 P.L.\ Lions and B.\ Mercier, Splitting algorithms for the sum of two
nonlinear operators.
\emph{SIAM Journal on Numerical Analysis}~16(6) (1979), 964--979. 

\bibitem{Meyer2000}
 C.D.\ Meyer, \emph{Matrix Analysis and Applied Linear Algebra},
SIAM, 2000.

\bibitem{Milne53}
W.E.\ Milne, \emph{Numerical
Solutions of Differential Equations},
New York, 1953.




\bibitem{Passty86}
G.B.\ Passty,
The parallel sum of nonlinear monotone operators,
\emph{Nonlinear Analysis}~10 (1986), 215--227.

\bibitem{Rock98}
R.T.\ Rockafellar and R.J-B.\ Wets,
\emph{Variational Analysis},
Springer-Verlag, 
corrected 3rd printing, 2009.
\bibitem{Simons1}
S.\ Simons,
\emph{Minimax and Monotonicity},
Springer-Verlag,
1998.
%
\bibitem{Simons2}
S.\ Simons,
\emph{From Hahn-Banach to Monotonicity},
Springer-Verlag,
2008.


\bibitem{Torii66}
T.\ Torii, Inversion of tridiagonal matrices and the stability of tridiagonal 
systems of linear equations, 
\emph{Information Processing in Japan}~6 (1966), 41--46.  

\bibitem{Thomee90}
V. Thom\'{e}e, Finite difference methods for 
linear parabolic equations, 
\emph{Handbook of numerical analysis} 1 (1990), 5--196.

\bibitem{Yama97}
T.\ Yamamoto and 
Y.\ Ikebe, Inversion of band matrices,
\emph{Linear Algebra and its Applications}~24 (1979),
105--111.

\bibitem{Zeidler2a}
E.\ Zeidler,
\emph{Nonlinear Functional Analysis and Its Applications II/A:
Linear Monotone Operators},
Springer-Verlag, 1990.

\bibitem{Zeidler2b}
E.\ Zeidler,
\emph{Nonlinear Functional Analysis and Its Applications II/B:
Nonlinear Monotone Operators},
Springer-Verlag, 1990.

\end{thebibliography}
}
\end{document}